\documentclass[a4paper,11pt]{article}
\usepackage{amsmath, amsfonts, amscd, amssymb, amsthm, enumerate, xypic}

\def\defterm{\textbf}
\def\id{\text{id}}

\newcommand{\Mat}{\operatorname{M}}
\newcommand{\Pgros}{\operatorname{\mathbb{P}}}
\newcommand{\GL}{\operatorname{GL}}
\newcommand{\Ker}{\operatorname{Ker}}
\newcommand{\Vect}{\operatorname{span}}
\newcommand{\im}{\operatorname{Im}}
\newcommand{\tr}{\operatorname{tr}}

\newcommand{\rk}{\operatorname{rk}}
\newcommand{\codim}{\operatorname{codim}}
\renewcommand{\setminus}{\smallsetminus}

\def\F{\mathbb{F}}
\def\K{\mathbb{K}}

\def\N{\mathbb{N}}

\def\calC{\mathcal{C}}
\def\calD{\mathcal{D}}

\def\calF{\mathcal{F}}
\def\calG{\mathcal{G}}

\def\calJ{\mathcal{J}}

\def\calM{\mathcal{M}}

\def\calP{\mathcal{P}}

\def\calR{\mathcal{R}}

\def\calV{\mathcal{V}}

\def\calX{\mathcal{X}}

\def\calZ{\mathcal{Z}}

\def\lcro{\mathopen{[\![}}
\def\rcro{\mathclose{]\!]}}

\theoremstyle{definition}
\newtheorem{Def}{Definition}
\newtheorem{Not}[Def]{Notation}

\theoremstyle{plain}
\newtheorem{theo}{Theorem}
\newtheorem{prop}[theo]{Proposition}

\newtheorem{lemme}[theo]{Lemma}
\newtheorem{claim}{Claim}

\theoremstyle{plain}

\theoremstyle{remark}
\newtheorem{Rems}{Remarks}
\newtheorem{Rem}[Rems]{Remark}

\title{The linear preservers of non-singularity in a large space of matrices}
\author{Cl\'ement de Seguins Pazzis\footnote{Professor of Mathematics at Lyc\'ee Priv\'e Sainte-Genevi\`eve, 2, rue
de l'\'Ecole des Postes, 78029 Versailles Cedex, FRANCE.}
\footnote{e-mail address: dsp.prof@gmail.com}}

\begin{document}

\thispagestyle{plain}

\maketitle

\begin{abstract}
Let $\K$ be an arbitrary (commutative) field, and $V$ be a linear subspace of $\Mat_n(\K)$ such that $\codim V<n-1$.
Using a recent generalization of a theorem of Atkinson and Lloyd \cite{dSPclass},
we show that every linear embedding of $V$ into $\Mat_n(\K)$ which strongly preserves non-singularity
must be $M \mapsto PMQ$ or $M \mapsto PM^TQ$ for some pair $(P,Q)$ of non-singular matrices of $\Mat_n(\K)$,
unless $n=3$, $\codim V=1$ and $\K \simeq \F_2$. This generalizes
a classical theorem of Dieudonn\'e with a similar strategy of proof. Weak linear preservers are also discussed,
as well as the exceptional case of a hyperplane of $\Mat_3(\F_2)$.
\end{abstract}

\vskip 2mm
\noindent
\emph{AMS Classification :} 15A86; 15A30

\vskip 2mm
\noindent
\emph{Keywords :} linear preservers, non-singular matrices, dimension, codimension.

\section{Introduction}

\subsection{Notations and goals}

Here, $\K$ will denote an arbitrary (commutative) field and $n$ a positive integer.
By a line in a vector space, we will always mean a $1$-dimensional \emph{linear} subspace of it.

We let $\Mat_{n,p}(\K)$ denote the set of matrices with $n$ rows, $p$ columns and entries in $\K$, and
$\GL_n(\K)$ the set of non-singular matrices in the algebra $\Mat_n(\K)$ of square matrices of order $n$.
The entries of a matrix $M \in \Mat_{n,p}(\K)$ are always denoted by small letters i.e.\ $M=(m_{i,j})$.
The rank of $M \in \Mat_{n,p}(\K)$ is denoted by $\rk M$.

We denote by $\frak{sl}_n(\K)$ the linear hyperplane of $\Mat_n(\K)$ consisting of matrices with trace zero.
We make the group $\GL_n(\K) \times \GL_p(\K)$ act on the set of linear subspaces of $\Mat_{n,p}(\K)$
by
$$(P,Q).V:=P\,V\,Q^{-1}.$$
Two linear subspaces of the same orbit will be called \textbf{equivalent}
(this means that they represent, in a change of basis, the same set of linear transformations from a $p$-dimensional vector space
to an $n$-dimensional vector space). \\
For $P$ and $Q$ in $\GL_n(\K)$, we define
$$u_{P,Q} : \begin{cases}
\Mat_n(\K) & \longrightarrow \Mat_n(\K) \\
M & \longmapsto P\,M\,Q
\end{cases}\quad \text{and} \quad v_{P,Q} :
\begin{cases}
\Mat_n(\K) & \longrightarrow \Mat_n(\K) \\
M & \longmapsto P\,M^T\,Q.
\end{cases}$$
Any map of the previous kind will be called a \defterm{Frobenius automorphism}.
It will be noteworthy to remark that the set of Frobenius automorphisms is a subgroup of the general linear group of the vector space $\Mat_n(\K)$.

\vskip 3mm
One of the earliest linear preserver problems was Dieudonn\'e's determination of the linear bijections $f$ of $\Mat_n(\K)$
which satisfy $f(\GL_n(\K)) \subset \GL_n(\K)$: using the structure of linear subspaces of singular matrices of $\Mat_n(\K)$
with maximal dimension, he showed that the solutions were precisely the Frobenius automorphisms
(see the recent \cite{dSPsinglin} for a full classification of non-invertible linear preservers).
More recently, the determination of the linear preservers of non-singularity was successfully carried out in many other contexts
(e.g.\ Banach spaces \cite{JafSour}, spaces of triangular matrices \cite{ChoiiLim}, spaces of symmetric matrices \cite{BeasleyLowey}).

Here, we wish to extend Dieudonn\'e's theorem to linear subspaces of $\Mat_n(\K)$ with a small codimension.
This question arose when we observed that a linear subspace of $\Mat_n(\K)$ is automatically spanned by its
non-singular elements provided its codimension is small enough (see Corollary 6 in \cite{dSPclass}).

\vskip 2mm
More precisely, we will prove the following results:

\begin{theo}\label{strong}
Let $V$ be a linear subspace of $\Mat_n(\K)$ such that $\codim V<n-1$.
Let $f : V \hookrightarrow \Mat_n(\K)$ be a linear embedding such that
$$\forall M \in V, \; f(M) \in \GL_n(\K) \Leftrightarrow M \in \GL_n(\K).$$
Then $f$ extends to a Frobenius automorphism of $\Mat_n(\K)$ unless
$n=3$, $\codim V=1$ and $\K \simeq \F_2$.
\end{theo}

The above theorem would normally be called a \emph{strong} linear preserver theorem. We will also prove the following two
theorems, which are more in tune with what the reader is used to (i.e.\ \emph{weak} linear preservers):

\begin{theo}\label{weakgeneral}
Let $V$ be a linear subspace of $\Mat_n(\K)$ such that $\codim V<n-1$.
Let $f : V \rightarrow V$ be a linear bijection such that $f\bigl(V \cap \GL_n(\K)\bigr) \subset \GL_n(\K)$.
Then $f$ extends to a Frobenius automorphism of $\Mat_n(\K)$ unless
$n=3$, $\codim V=1$ and $\K \simeq \F_2$.
\end{theo}

\begin{theo}\label{weakKinfinite}
Assume $\K$ is infinite.
Let $V$ be a linear subspace of $\Mat_n(\K)$ such that $\codim V<n-1$,
and $f : V \hookrightarrow \Mat_n(\K)$ be a linear embedding such that $f\bigl(V \cap \GL_n(\K)\bigr) \subset \GL_n(\K)$.
Then $f$ extends to a Frobenius automorphism of $\Mat_n(\K)$.
\end{theo}

Whether the last theorem still holds for finite fields remains an exciting open problem.

\vskip 2mm
Before proving those results, we wish to show that the upper bound $n-1$ is tight
provided $n \geq 3$ (the case $n=2$ and $\codim V=1$ will be dealt with in Section \ref{hyper}).
Consider indeed the subspace
$$H_n:=\biggl\{
\begin{bmatrix}
M & C \\
0 & a
\end{bmatrix} \mid (M,C,a) \in \Mat_{n-1}(\K) \times \Mat_{n-1,1}(\K) \times \K \biggr\}$$
and the linear bijection:
$$\Phi : \begin{bmatrix}
M & C \\
0 & a
\end{bmatrix} \mapsto \begin{bmatrix}
M & C+m_{2,2}.e_1 \\
0 & a
\end{bmatrix}$$
where $e_1:=\begin{bmatrix}
1 &
0 &
\cdots &
0
\end{bmatrix}^T$ (and of course $m_{2,2}$ is $M$'s entry at the $(2,2)$ spot). Since the matrix $\begin{bmatrix}
M & C \\
0 & a
\end{bmatrix}$ is non-singular if and only if $M$ is non-singular and $a \neq 0$,
it follows that $\Phi$ is a strong preserver of non-singularity.
However, $\Phi$ does not extend to a Frobenius automorphism of $\Mat_n(\K)$ since it is not a rank preserver:
indeed, taking $M=E_{2,2}$ (the matrix with entry $1$ at the spot $(2,2)$, and zero entries elsewhere), one has
$\rk \begin{bmatrix}
M & 0 \\
0 & 0
\end{bmatrix}=1$ whereas $\rk \Phi \begin{bmatrix}
M & 0 \\
0 & 0
\end{bmatrix}=2$.

\subsection{Strategy of proof and structure of the article}

Our strategy for the proof of Theorem \ref{strong} is essentially similar to that of Dieudonn\'e \cite{Dieudonne}:
given a linear embedding $f : V \hookrightarrow \Mat_n(\K)$ which strongly preserves non-singularity, we study
the preimages of subspaces of singular matrices of $\Mat_n(\K)$ with maximal dimension. To understand the structure of those preimages,
we will use our recent generalization \cite{dSPclass} of a theorem of Atkinson and Lloyd \cite{AtkLloyd}.
From there, we will show (leaving aside a technical problem in the case $\codim V=n-2$, which
will be tackled in Section \ref{techlemma}) that the situation may be reduced to the one where $f$
preserves the image of any matrix of $V$.
We will then use the so-called representation lemma of \cite{dSPclass} (Theorem 8) to show that this property forces $f$
to have the form $M \mapsto M\,Q$ for some $Q \in \GL_n(\K)$, which will conclude the proof.
In Section \ref{weakproof}, we will derive Theorems \ref{weakgeneral}
and \ref{weakKinfinite} from Theorem \ref{strong}: this is trivial in the case of a finite field,
and will involve considerations of polynomials in the case $\K$ is infinite (we will prove that
every polynomial on $V$ which vanishes on its singular elements must be a multiple of the determinant restricted to $V$:
this will show that the weak preservation of non-singularity implies the strong one for a one-to-one linear map).

\vskip 2mm
The remaining two sections will be devoted to the inspection of special cases:
\begin{itemize}
\item In Section \ref{F2}, we will show that there is a linear hyperplane $V$ of $\Mat_3(\F_2)$ and an embedding
which do not satisfy the conclusion of Theorem \ref{weakgeneral}, and we will also determine
which linear hyperplanes of $\Mat_3(\F_2)$ do satisfy this conclusion for any embedding. Naturally, this is related to the special
case in the generalized Atkinson-Lloyd theorem, see Theorem 2 of \cite{dSPclass}.
\item In Section \ref{hyper}, we will show that the conclusions of Theorems \ref{strong} to \ref{weakKinfinite} still hold
in the case $n=2$ and $V$ is a linear hyperplane of $\Mat_2(\K)$. This is interesting because
it shows that the result holds for linear hyperplanes regardless of $n$, e.g. for $\frak{sl}_n(\K)$
(in that case, even if $\K \simeq \F_2$, see Section \ref{F2}).
\end{itemize}

\section{Preimage of large singular subspaces}\label{strongstart}

\subsection{A review of large subspaces of singular matrices}

\begin{Def}
A linear subspace $V$ of $\Mat_n(\K)$ is called \textbf{singular} when all its matrices are singular.
It is said to have rank $k$ when $k=\max\bigl\{\rk M \mid M \in V\}$.
\end{Def}

\begin{Not}
We set $E:=\K^n$ and let $\Pgros(E)$ denote the projective space associated to $E$, i.e.\
the set of lines in $E$.
We equip $E$ with the non-degenerate symmetric bilinear form $(X,Y) \mapsto X^T\,Y$.
Given $D \in \Pgros(E)$, the linear hyperplane $D^\bot=\{X \in \K^n : \; X^T D=0\}$ is the annihilator of $D$, and we set
$$\calM_D:=\bigl\{M \in \Mat_n(\K) : \; D \subset \Ker M\bigr\} \quad \text{and} \quad
\calM^D:=\bigl\{M \in \Mat_n(\K) : \; \im M \subset D^\bot\bigr\}.$$
\end{Not}

\begin{Rem}
Notice that $\calM_D^T=\calM^D$ and\footnote{For $V \subset \Mat_n(\K)$, we write $V^T:=\{M^T\mid M \in V\}$.} $(\calM^D)^T=\calM_D$, and that $\calM_D$ and $\calM^D$
are singular subspaces of $\Mat_n(\K)$ with codimension $n$. Classically (see \cite{Dieudonne}, or prove it directly),
these are maximal singular subspaces of $\Mat_n(\K)$ (i.e.\ maximal in the set of the singular subspaces of $\Mat_n(\K)$, ordered by
inclusion).
\end{Rem}

\begin{Not}
Let $(s,t)\in \lcro 0,n\rcro \times \lcro 0,p\rcro$.
Set then
$$\calR(s,t):=\biggl\{\begin{bmatrix}
M & N \\
P & 0
\end{bmatrix} \mid M \in \Mat_{s,t}(\K), \; N \in \Mat_{s,p-t}(\K), P \in \Mat_{n-s,t}(\K)\biggr\} \subset \Mat_{n,p}(\K)$$
(notice that we understate $n$ and $p$ in this notation; however, no confusion should arise when we use it).
\end{Not}

With the above notations, we may reformulate a theorem of Atkinson and Lloyd \cite{AtkLloyd} recently generalized in
\cite{dSPclass} to an arbitrary field:

\begin{theo}\label{atklloydtheorem}
Let $V$ be a singular subspace of $\Mat_n(\K)$ such that $\codim V \leq 2n-2$.
Then one and only one of the following three conditions holds, unless $n=3$, $\codim V=1$ and $\# \K=2$:
\begin{enumerate}[(i)]
\item $V \subset \calM_D$ for a unique $D \in \Pgros(E)$;
\item $V \subset \calM^D$ for a unique $D \in \Pgros(E)$;
\item $\codim V=2n-2$ and $V$ is equivalent to $\calR(n-2,1)$ or to $\calR(1,n-2)$.
\end{enumerate}
\end{theo}

\begin{Rem}
In \cite{dSPclass}, the incompatibility between (i) and (ii) was not proven, nor was the uniqueness of $D$ in the case
$V$ is equivalent to a subspace of $\calR(n-1,0)$ or $\calR(0,n-1)$.
However, the proof is essentially similar to that of \cite{AtkLloyd}.
\end{Rem}

\subsection{Reduction to the case of an image-preserving map}\label{red1}

In this paragraph, we let $V$ be a linear subspace of $\Mat_n(\K)$ with codimension lesser than $n-1$,
and $f : V \hookrightarrow \Mat_n(\K)$ be a linear embedding such that
$f^{-1}(\GL_n(\K))=V \cap \GL_n(\K)$. We discard the case $n=3$, $\codim V=1$ and $\# \K=2$.
We also assume $n \geq 3$, since $V=\Mat_2(\K)$ if $n=2$, in which case the result we claim is already known (see \cite{Dieudonne}).
Our aim is to prove that, by pre and post-composing $f$ with well-chosen Frobenius automorphisms,
we may obtain a linear map (necessarily one-to-one) which preserves the image for any matrix of $V$.
Following Dieudonn\'e \cite{Dieudonne}, the basic idea is to study the subspaces
$f^{-1}(\calM_D)$ and $f^{-1}(\calM^D)$ for every $D \in \Pgros(E)$.

Let $D \in \Pgros(E)$. Then $\calM_D$ has codimension $n$ in $\Mat_n(\K)$, hence the rank theorem shows
that $\codim_V f^{-1}(\calM_D)=\codim_{f(V)} f(V) \cap \calM_D \leq \codim_{\calM_n(\K)} \calM_D=n$,
hence $\codim_{\Mat_n(\K)} f^{-1}(\calM_D) \leq 2n-2$ since $\codim_{\Mat_n(\K)} V \leq n-2$.
However, since $\calM_D$ is a maximal singular subspace of $\Mat_n(\K)$, $f$ is one-to-one and $f^{-1}(\GL_n(\K))=V \cap \GL_n(\K)$,
it is clear that $f^{-1}(\calM_D)$ is a maximal singular subspace of $V$.
A similar argument shows that $f^{-1}(\calM^D)$ has the same properties, hence the following result:

\begin{claim}\label{basicclaim}
For every $D \in \Pgros(E)$, the linear subspaces $f^{-1}(\calM_D)$ and $f^{-1}(\calM^D)$ are maximal singular subspaces of
$V$ with codimension $\leq 2n-2$ in $\Mat_n(\K)$.
\end{claim}

Using Theorem \ref{atklloydtheorem}, we deduce:

\begin{claim}\label{basiclaim2}
For any $D \in \Pgros(E)$, one and only one of the following conditions holds:
\begin{enumerate}[(i)]
\item There is a unique $D' \in \Pgros(E)$ such that $f^{-1}(\calM_D)=V \cap \calM_{D'}$;
\item There is a unique $D' \in \Pgros(E)$ such that $f^{-1}(\calM_D)=V \cap \calM^{D'}$;
\item The subspace $f^{-1}(\calM_D)$ is equivalent to $\calR(n-2,1)$ or $\calR(1,n-2)$, and $\codim V=n-2$.
\end{enumerate}
A similar result also holds for $f^{-1}(\calM^D)$ instead of $f^{-1}(\calM_D)$.
\end{claim}

For the rest of the paragraph, we will admit the following lemma, the proof of which is tedious
and will only be given in Section \ref{techlemma}:

\begin{lemme}\label{hightechlemma}
Let $V$ be a linear subspace of codimension $n-2$ in $\Mat_n(\K)$, and $g : V \hookrightarrow \Mat_n(\K)$
be a linear embedding such that $g^{-1}(\GL_n(\K))=V \cap \GL_n(\K)$.
Assume that $(n,\# \K) \neq (3,2)$.
Let $D \in \Pgros(E)$. Then $g^{-1}(\calM_D)$ is equivalent neither to $\calR(n-2,1)$ nor to $\calR(1,n-2)$.
\end{lemme}

This yields:

\begin{claim}\label{inverseimageclaim}
For every $D \in \Pgros(E)$, there is a unique $D' \in \Pgros(E)$ such that
$f^{-1}(\calM_D)=V \cap \calM_{D'}$ or $f^{-1}(\calM_D)=V \cap \calM^{D'}$, and only one of those two conditions holds.
\end{claim}

Here is our next claim:

\begin{claim}\label{coherenceclaim1}
Assume there is a pair $(D_1,D'_1) \in \Pgros(E)^2$ such that $f^{-1}(\calM_{D_1})=V \cap \calM_{D_1'}$.
Then, for every $D \in \Pgros(E)$, there is a unique $D' \in \Pgros(E)$ such that $f^{-1}(\calM_D)=V \cap \calM_{D'.}$
\end{claim}

\begin{proof}
Let $D_2 \in \Pgros(E) \setminus \{D_1\}$. We may then choose non-zero vectors $x_1 \in D_1$, $x_2 \in D_2$
and extend $(x_1,x_2)$ into a basis $(x_1,\dots,x_n)$ of $E$. Set $D_i:=\Vect(x_i)$ for $i \in \lcro 3,n\rcro$.
For every $i \in \lcro 2,n\rcro$, we may find a (unique) $D'_i \in \Pgros(E)$ such that
$f^{-1}(\calM_{D_i})=V \cap \calM_{D'_i}$ or $f^{-1}(\calM_{D_i})=V \cap \calM^{D'_i}$.
Define $I$ as the set of those $i \in \lcro 1,n\rcro$ such that $f^{-1}(\calM_{D_i})=V \cap \calM_{D'_i}$, and $J:=\lcro 1,n\rcro \setminus I$.
Set also $F:=\underset{i \in I}{\sum} D'_i$ and $G:=\underset{i \in J}{\sum} D'_i$, and notice that $\dim F+\dim G \leq n$.
Note that $\underset{1 \leq i \leq n}{\bigcap} \calM_{D_i}=\{0\}$, hence
$$\{0\}=f^{-1}\biggl(\underset{1 \leq i \leq n}{\bigcap} \calM_{D_i}\biggr)
=\underset{1 \leq i \leq n}{\bigcap} V\cap f^{-1}(\calM_{D_i})=
V \cap \underset{i \in I}{\bigcap} \calM_{D'_i} \cap \underset{j \in J}{\bigcap} \calM^{D'_j}.$$
However, $\underset{i \in I}{\bigcap} \calM_{D'_i}$ is the set of matrices $M \in \Mat_n(\K)$ such that
$F \subset \Ker M$, and $\underset{j \in J}{\bigcap} \calM^{D'_j}$ the set of matrices $M \in \Mat_n(\K)$ such that
$\im M \subset G^\bot$, hence
$$\dim \Biggl[\underset{i \in I}{\bigcap} \calM_{D'_i} \cap \underset{j \in J}{\bigcap} \calM^{D'_j}\Biggr]
=(n-\dim F)\,(n-\dim G) \geq \dim G\,(n-\dim G).$$
Assume finally that $J \neq \emptyset$. Then $1 \leq \dim G \leq n-1$ hence
$(\dim G)\,(n-\dim G) \geq n-1$. Since $\codim V <n-1$, this yields
$$V \cap \underset{i \in I}{\bigcap} \calM_{D'_i} \cap \underset{j \in I}{\bigcap} \calM^{D'_j} \neq \{0\},$$
in contradiction with a previous result. We deduce that $J=\emptyset$.
\end{proof}

With a similar proof, or by applying the above results to $M \mapsto f(M^T)^T$, we also have:

\begin{claim}\label{coherenceclaim2}
Assume there is a pair $(D_1,D'_1) \in \Pgros(E)^2$ such that $f^{-1}(\calM^{D_1})=V \cap \calM^{D_1'}$.
Then, for every $D \in \Pgros(E)$, there is a unique $D' \in \Pgros(E)$ such that $f^{-1}(\calM^D)=V \cap \calM^{D'.}$
\end{claim}

We now lose no generality making the following additional assumption:
\begin{center}
There is a pair $(D_1,D'_1) \in \Pgros(E)^2$ such that $f^{-1}(\calM^{D_1})=V \cap \calM^{D'_1}$.
\end{center}
Indeed, in the case this does not hold, we still have some pair $(D_1,D'_1) \in \Pgros(E)^2$ such that
$f^{-1}(\calM_{D_1})=V \cap \calM^{D'_1}$, and we may then replace $f$ with $M \mapsto f(M)^T$, which satisfies the preceding assumption.

Now, Claim \ref{coherenceclaim1} applied to both $f$ and $f^{-1} : f(V) \hookrightarrow \Mat_n(\K)$
shows there is a bijective map $\varphi : \Pgros(E) \rightarrow \Pgros(E)$ such that
$f(V \cap \calM^D)=f(V) \cap \calM^{\varphi(D)}$ for every $D \in \Pgros(E)$.
Let $D \in \Pgros(E)$. If $f(V \cap \calM_D)=f(V) \cap \calM^{D'}$ for some line $D'$,
then $f(V \cap \calM_D)=f(V \cap \calM_{\varphi^{-1}(D')})$ and therefore $V \cap \calM_D=V \cap \calM_{\varphi^{-1}(D')}$, contradicting
the uniqueness in Theorem \ref{atklloydtheorem}. Therefore $f(V \cap \calM_D)=V \cap \calM_{D'}$ for some $D' \in \Pgros(E)$.
Claim \ref{coherenceclaim1} applied to both $f$ and $f^{-1}$ then shows there is a bijective map
$\psi : \Pgros(E) \rightarrow \Pgros(E)$ such that $f(V \cap \calM_D)=f(V) \cap \calM_{\psi(D)}$ for every $D \in \Pgros(E)$.

\begin{claim}
The map $\varphi$ is a projective automorphism of $\Pgros(E)$.
\end{claim}

\begin{proof}
First notice that $\varphi$ preserves alinement on the projective space $\Pgros(E)$.
Indeed, let $D_1$, $D_2$ and $D_3$ be three distinct lines of $E$ and assume that
$D_1+D_2+D_3$ has dimension $2$ and $\varphi(D_1)+\varphi(D_2)+\varphi(D_3)$ has dimension $3$.
Notice that $\underset{i=1}{\overset{3}{\bigcap}}\calM^{D_i}$ has codimension $2n$ in $\Mat_n(\K)$ whereas
$\underset{i=1}{\overset{3}{\bigcap}}\calM^{\varphi(D_i)}$ has codimension $3n$.
It follows that
$$\dim \biggl[f(V) \cap \underset{i=1}{\overset{3}{\bigcap}}\calM^{\varphi(D_i)}\biggr] \leq n(n-3),$$
whereas
$$\dim \biggl[V \cap \underset{i=1}{\overset{3}{\bigcap}}\calM^{D_i}\biggr] \geq n(n-2)-n+2>n(n-3),$$
contradicting the definition of $\varphi$. \\
By the fundamental theorem of projective geometry (recall that $\dim E \geq 3$), we deduce that there is a semi-linear automorphism
$u$ of $E$ such that $\varphi(D)=u(D)$ for every $D \in \Pgros(E)$. The same line of reasoning shows there is a semi-linear automorphism
$v$ of $E$ such that $\psi(D)=v(D)$ for every $D \in \Pgros(E)$. \\
It only remains to prove that $u$ is linear.
Consider an arbitrary non-zero vector $Y_0 \in E \setminus \{0\}$, notice that
$\{XY_0^T\mid X \in E\}$ is an $n$-dimensional linear subspace of $\Mat_n(\K)$, hence we may find two
linearly independent vectors $X_1$ and $X_2$ in $E$ such that $X_1Y_0^T$ and $X_2Y_0^T$ belong to $V$.
Since $v$ is a semi-linear automorphism of $E$,
we find that there is a non-zero vector $Y'_0 \in E$ such that, for every $X \in E$ such that
$XY_0^T \in V$, one has $f(XY_0^T) =X'(Y'_0)^T$ for some $X' \in E$:
indeed, we may consider a basis $(Y_2,\dots,Y_n)$ of the linear hyperplane $\{Y_0\}^\bot$, notice then that
$\underset{i=2}{\overset{n}{\bigcap}}\calM_{v(\Vect(Y_i))}$ is the set of matrices which vanish on
the hyperplane $\Vect(v(Y_i))_{2 \leq i \leq n}$ and then choose a non-zero vector $Y'_0$ in its orthogonal subspace.
We recover two non-zero vectors $X'_1$ and $X'_2$ such that $f(X_1 Y_0^T)=X'_1 (Y_0')^T$ and $f(X_2Y_0^T)=X'_2 (Y_0')^T$.
Let now $(\alpha,\beta) \in \K^2$. Then $f((\alpha X_1+\beta X_2)Y_0^T)=(\alpha X'_1+\beta X'_2)(Y'_0)^T$
since $f$ is linear. We deduce that
$\alpha X'_1+\beta X'_2$ is orthogonal to $u(X)$ for every $X$ orthogonal to $\alpha X_1+\beta X_2$.
We then choose two linearly independent vectors $Z_1$ and $Z_2$ in $E$
such that $X_i^T\,Z_j=\delta_{i,j}$ for every $(i,j)\in \{1,2\}^2$, and let
$\lambda : \K \rightarrow \K$ denote the field automorphism associated to the semi-linear map $u$.
Then $\alpha X'_1+\beta X'_2$ is orthogonal to $u(\beta Z_1-\alpha Z_2)=\lambda(\beta)\,u(Z_1)-\lambda(\alpha)\,u(Z_2)$.
In particular, $X'_1 \bot u(Z_2)$ and $X'_2 \bot u(Z_1)$. Taking $\beta=1$ then shows that
$\alpha\,(X'_1)^T u(Z_1)=\lambda(\alpha)\,(X'_2)^T\,u(Z_2)$, and the special case $\alpha=1$ then yields:
$$\forall \alpha \in \K, \; (\alpha-\lambda(\alpha))(X'_1)^T\, u(Z_1)=0.$$
Notice finally that $X_1Y_0^T \not\in \calM^{\Vect(Z_1)}$ hence $X'_1(Y'_0)^T \not\in \calM^{\Vect(u(Z_1))}$
which shows that $(X'_1)^T\, u(Z_1) \neq 0$. We deduce that $\lambda=\id_\K$.
\end{proof}

Denote then by $P$ the non-singular matrix of $\Mat_n(\K)$ such that
$\varphi(X)=PX$ for every $X \in E$. Then the map
$f' : M \mapsto P^T\,f(M)$ satisfies all the assumptions of Theorem \ref{strong} with the additional property:
\begin{center}
For every $D \in \Pgros(E)$, one has $f'(V \cap \calM^D)=f'(V) \cap \calM^D$.
\end{center}

We may now conclude this section by summing up the above results, still assuming Lemma
\ref{hightechlemma} holds:

\begin{prop}
Let $V$ be a linear subspace of $\Mat_n(\K)$ such that $\codim V<n-1$ and $n \geq 3$.
Let $f : V \hookrightarrow \Mat_n(\K)$ be a linear embedding such that
$$\forall M \in V, \quad f(M) \in \GL_n(\K) \Leftrightarrow M \in \GL_n(\K).$$
Unless $(n,\codim V,\# \K)=(3,1,2)$, there are two Frobenius automorphisms\footnote{One
may even take $(u,V')=(\id_V,V)$ or $(u,V')=(X \mapsto X^T,V^T)$, whilst
$v : M \mapsto QM$ for some $Q \in \GL_n(\K)$.}
$u$ and $v$, together with a linear subspace $V'$ of $\Mat_n(\K)$ with $\dim V'=\dim V$,
and a linear embedding $f' : V' \hookrightarrow \Mat_n(\K)$ such that:
\begin{enumerate}[(i)]
\item $u(V)=V'$;
\item $f=v \circ f' \circ u_{|V}$;
\item for every $M \in V'$, one has $\im f'(M)=\im M$.
\end{enumerate}
\end{prop}

\subsection{Image-preserving linear embeddings}\label{kerandim}

We will now prove the following result, which completes the proof of Theorem \ref{strong}
\emph{modulo} the proof of Lemma \ref{hightechlemma}.

\begin{prop}\label{improp}
Let $V$ be a linear subspace of $\Mat_n(\K)$ such that $\codim V<n-1$.
Let $f : V \hookrightarrow \Mat_n(\K)$ be a linear embedding such that
$\im f(M)=\im M$ for every $M \in V$.
Then $f$ coincides with $u_{I_n,Q}$ on $V$ for some $Q \in \GL_n(\K)$.
\end{prop}

This is direct consequence of the following lemma, which was recently proven in \cite{dSPclass}
(Theorem 8):

\begin{lemme}[Representation lemma]\label{reprtheorem}
Let $(n,p,r)\in \N^3$. Let $V$ be a linear subspace of
$\Mat_{n,r}(\K)$ such that $\dim V \geq nr-n+2$. Let
$\varphi : V \rightarrow \Mat_{n,p}(\K)$ be a linear map such that
$\im \varphi(M) \subset \im M$ for every $M \in V$. \\
Then there exists $C \in \Mat_{r,p}(\K)$ such that
$\forall M \in V, \; \varphi(M)=MC$.
\end{lemme}

\begin{proof}[Proof of Proposition \ref{improp}]
Applying Lemma \ref{reprtheorem} to $V$ and $f$ (with $p=r=n$), we find a matrix $Q \in \Mat_n(\K)$
such that $\forall M \in V, \; f(M)=MQ$. \\
Since $f$ is one-to-one, we deduce that $V$ contains no non-zero matrix which vanishes on $\im Q$.
If $\im Q \subsetneq \K^n$, this would yield $\codim_{\Mat_n(\K)} V \geq n$,  contradicting our assumptions. Therefore $Q$ is non-singular
and $f=u_{I_n,Q}$.
\end{proof}

\section{A (very) technical lemma}\label{techlemma}

This entire section is devoted to the proof of Lemma \ref{hightechlemma}, which is the last
obstacle for proving Theorem \ref{strong}. In the whole proof, we denote by $(e_1,\dots,e_n)$ the canonical basis of $E=\K^n$.

\subsection{Starting the proof}

We use a \emph{reductio ad absurdum}. Let $V$ and $g$ be as in Lemma \ref{hightechlemma},
and assume that there is a line $D$ such that $g^{-1}(\calM_D)$ is equivalent to $\calR(1,n-2)$
(notice that the case where $g^{-1}(\calM_D)$ is equivalent to $\calR(n-2,1)$ may be reduced to this one by pre-composing $g$
with $M \mapsto M^T$). By composing $g$ with a well-chosen Frobenius automorphism, we may also
assume that $D=\Vect(e_n)$. We finally lose no generality assuming that
\begin{equation}\label{H0}
g^{-1}(\calM_D)=\calR(1,n-2).
\end{equation}
To make things clearer, those assumptions mean that:
$V$ contains every matrix of the form
$M=\begin{bmatrix}
? & a & b \\
N & 0 & 0
\end{bmatrix}$ with $(a,b)\in \K^2$ and $N \in \Mat_{n-1,n-2}(\K)$, for such a matrix $M$ we always have
$g(M)=\begin{bmatrix}
N' & 0 \\
\end{bmatrix}$ for some $N'\in \Mat_{n,n-1}(\K)$, and $\calR(1,n-2)$ is precisely the set of matrices
in $V$ whose images by $g$ have $0$ as last column.

\begin{Not}
In the rest of the proof, we set $V':=\calR(0,n-2)$, i.e.\ $V'$ is the set of all matrices of $\Mat_n(\K)$ with
zero as $(n-1)$-th and $n$-th column.
\end{Not}

We will first investigate the structure of $g(V \cap \calM_{D_1})$ for an arbitrary line $D_1 \subset \Vect(e_{n-1},e_n)$.

\subsection{Sorting out the structure of $g(V \cap \calM_{D_1})$ (I)}\label{sort1}

\begin{Not}
For an arbitrary line $D_1 \subset \Vect(e_{n-1},e_n)$, we set
$$H_{D_1}:=g(V \cap \calM_{D_1}).$$
\end{Not}

\begin{claim}\label{different}
Let $D_1$ and $D_2$ be distinct lines in $\Pgros(\Vect(e_{n-1},e_n))$.
Then $H_{D_1} \neq H_{D_2}$.
\end{claim}

\begin{proof}
Notice indeed that $V'=V \cap \calM_{D_1} \cap \calM_{D_2}$ has a codimension greater than or equal to $2n$ in $\Mat_n(\K)$.
On the other hand, we know from the inclusion $\calR(1,n-2) \subset V$ that $V \cap \calM_{D_1}$ has a codimension lesser than
$2n$ in $\Mat_n(\K)$,
which shows $V \cap \calM_{D_1} \cap \calM_{D_2} \neq V \cap \calM_{D_1}$ and proves our claim since $g$ is one-to-one.
\end{proof}

\begin{claim}\label{impossclaim1}
Let $D_1$ be a line of $\Vect(e_{n-1},e_n)$. Then there is no line $D'_1$ such that
$H_{D_1} \subset \calM_{D'_1}$ or $H_{D_1} \subset \calM^{D'_1}$.
\end{claim}

\begin{proof}
Assumption \eqref{H0}, together with the conclusions of the present claim, are unchanged should we choose $P \in \GL_n(\K)$
which leaves $\Vect(e_1,\dots,e_{n-2})$ and $\Vect(e_{n-1},e_n)$ invariant
and replace $V$ and $g$ respectively with $VP^{-1}$ and $g \circ u_{I_n,P}$.
Therefore we lose no generality assuming that $D_1=\Vect(e_n)$. In this case, we use
a \emph{reductio ad absurdum} and assume there is a line $D'_1$ such that
$H_{D_1} \subset \calM_{D'_1}$ or
$H_{D_1} \subset \calM^{D'_1}$. \\
Assume first that
$H_{D_1} \subset \calM_{D'_1}$. If $D'_1=D$, then $V \cap \calM_{D_1} \subset g^{-1}(\calM_D)=\calR(1,n-2)$,
and we deduce that every matrix of $V \cap \calM_{D_1}$ has zero as last column and, starting from the second one,
all its entries on the $(n-1)$-th last column are zero, therefore
$\codim_{\Mat_n(\K)} (V \cap \calM_{D_1}) \geq 2n-1$, which contradicts the assumption $\codim_{\Mat_n(\K)} V \leq n-2$.
Therefore $D'_1 \neq D$, which yields $\dim (\calM_D \cap \calM_{D'_1})=n(n-2)$,
and it follows that $\dim\bigl(g(V) \cap \calM_D \cap \calM_{D'_1}\bigr) \leq  n(n-2)$.
However $V \cap \calR(1,n-2) \cap \calM_{D_1}=\calR(1,n-2) \cap \calM_{D_1}$ has dimension $n(n-2)+1$, which contradicts the fact that $g$ is one-to-one. \\
We deduce that $H_{D_1} \subset \calM^{D'_1}$, and we lose no generality assuming that $D'_1=\Vect(e_n)$.
In particular, any matrix of $V'$ is mapped by $g$ to
$\begin{bmatrix}
N & 0 \\
0 & 0
\end{bmatrix}$ for some $N \in \Mat_{n-1}(\K)$. \\
Pick now a second line $D_2 \subset \Vect(e_{n-1},e_n)$
(different from $D_1$) and consider the subspace $g(V \cap \calM_{D_2})$:
by the previous line of reasoning, it may not be included in any $\calM_{D'_2}$.
Assume it is included in $\calM^{D'_2}$ for some line $D'_2 \subset E$.
Then $D'_2$ must be different from $D'_1$ by Claim \ref{different}. Then, since $\calR(1,n-2) \subset V$, one has
$$g(V')=g(\calM_{D_1} \cap \calM_{D_2} \cap \calR(1,n-2))
\subset \calM^{D'_1} \cap \calM^{D'_2} \cap \calM_D$$
which shows that $\codim g(V') \geq n+2(n-1)>2n$, contradicting $\codim V'=2n$. \\
Now, we may apply Claim \ref{basiclaim2} to the map $g^{-1}$, and deduce that $H_{D_2}$ is equivalent either to $\calR(1,n-2)$ or $\calR(n-2,1)$. \\
Assume first that $H_{D_2}$ is equivalent to $\calR(1,n-2)$.
Then there is a $2$-dimensional subspace $P$ of $E$ such that, for every
$x \in P$, one has $\dim H_{D_2}x \leq 1$ (where $H_{D_2}x=\bigl\{Mx \mid M \in H_{D_2}\bigr\}$).
We may then choose a non-zero vector $x$ in $P \cap \Vect(e_1,\dots,e_{n-1})$, which shows that
$$\codim(H_{D_2} \cap \calM^{D'_1} \cap \calM_D) \geq (n-2)+(n-1)+n=3n-3.$$
However, since $\calR(1,n-2) \subset V$,
$$H_{D_2} \cap \calM^{D'_1} \cap \calM_D=g(\calM_{D_2} \cap \calM_{D_1} \cap \calR(1,n-2))=g(V')$$
hence $H_{D_2} \cap \calM^{D'_1} \cap \calM_D$ has codimension $2n$ in $\Mat_n(\K)$.
Notice that a similar line of reasoning holds in the case $H_{D_2}$ is equivalent to $\calR(n-2,1)$,
so this yields a contradiction if $n>3$. \\
Assume finally that $n=3$. In this case, we lose no generality assuming that $e_2$ belongs to
$P$. Then $H_{D_2}e_2=\Vect(y)$ for some $y \in \K^3 \setminus \{0\}$, and
$H_{D_2}$ contains the $3$-dimensional space $\calZ$ of all matrices $M \in \Mat_3(\K)$ such that $\im(M) \subset \Vect(y)$.
However $g(V') \subset H_{D_1} \subset \calM^{\Vect(e_3)}$.
If $\Vect(y)=\Vect(e_3)$, then we would have $V' \cap g^{-1}(\calZ)=\{0\}$, which is not possible since
$V'$ and $g^{-1}(\calZ)$ are both $3$-dimensional subspaces of the $5$-dimensional space $V \cap \calM_{D_2}$. \\
Therefore we lose no generality assuming that $H_{D_2}e_2=\Vect(e_1)$, in which case we find that $g$ maps any matrix of the form
$\begin{bmatrix}
? & 0 & 0 \\
? & 0 & 0 \\
? & 0 & 0
\end{bmatrix}$ to a matrix of the form
$\begin{bmatrix}
? & ? & 0 \\
? & 0 & 0 \\
0 & 0 & 0
\end{bmatrix}$. Set
$$G:=\Biggl\{\begin{bmatrix}
a & b & 0 \\
c & 0 & 0 \\
0 & 0 & 0
\end{bmatrix} \mid (a,b,c)\in \K^3\Biggr\}.$$
Let $D_3 \in \Pgros(\Vect(e_2,e_3)) \setminus \{D_1\}$. Then $H_{D_3}$ contains
$g(V')=G$. However, the previous considerations apply with $D_2$ replaced by $D_3$, hence $H_{D_3}$ is equivalent to $\calR(1,1)$.
The fact that $G \subset H_{D_3}$ then yields $H_{D_3}=\calR(1,1)$:
indeed, $G$ has two obvious $2$-dimensional linear subspaces of rank $1$, their sum has rank $2$, so each one
must be contained in one and only one of the two $3$-dimensional rank $1$ linear subspaces of $H_{D_3}$, which
forces those subspaces to be
$$\Biggl\{\begin{bmatrix}
a & b & c \\
0 & 0 & 0 \\
0 & 0 & 0
\end{bmatrix} \mid (a,b,c)\in \K^3\biggr\} \quad \text{and} \quad
\Biggl\{\begin{bmatrix}
a & 0 & 0 \\
b & 0 & 0 \\
c & 0 & 0
\end{bmatrix} \mid (a,b,c)\in \K^3\Biggr\}.$$
Finally, if we choose $D_3$ different from $D_2$, then
we recover $H_{D_3}=\calR(1,1)=H_{D_2}$, contradicting Claim \ref{different}.
\end{proof}

\subsection{Sorting out the structure of $g(V \cap \calM_{D_1})$ (II)}

Applying Claim \ref{basiclaim2} to $g^{-1}$, we deduce from Claim \ref{impossclaim1}:

\begin{claim}\label{structureclaim}
For any line $D_1 \subset \Vect(e_{n-1},e_n)$, the subspace
$H_{D_1}$ is equivalent to $\calR(1,n-2)$ or to $\calR(n-2,1)$.
\end{claim}

We now prove:

\begin{claim}\label{rightside}
Let $D_1 \subset \Vect(e_{n-1},e_n)$ be a line. Then
$H_{D_1}$ is equivalent to $\calR(1,n-2)$.
\end{claim}

\begin{proof}
This follows directly from the preceding claim when $n=3$. Assume now that $n \geq 4$.
As in the beginning of the proof of Claim \ref{impossclaim1}, we lose no generality assuming that $D_1=\Vect(e_n)$.
We use another \emph{reductio ad absurdum} by assuming that $H_{D_1}$ is equivalent to $\calR(n-2,1)$.
Then there is a unique linear subspace $F$ with codimension $2$ in $E$ such that
$\forall x \in E \setminus \{0\}, \; H_{D_1} x=F\; \text{or}\; H_{D_1}x=E$.
We lose no generality assuming
$F=\Vect(e_1,\dots,e_{n-2})$ (we may reduce the general case to this one by composing $g$
with $u_{P,I_n}$ for some well-chosen $P \in \GL_n(\K)$).
Then the set of vectors $x \in E$ such that $H_{D_1}x \subset F$
is a linear hyperplane $G$ of $E$. \\
We may then find a linear subspace $G' \subset G$ such that $G' \cap \Vect(e_n)=\{0\}$ and $\dim G'=n-2$.
It then easily follows that $g(\calM_{D_1} \cap \calR(1,n-2))$ has a codimension greater than or equal to
$n+2(n-2)$ in $\Mat_n(\K)$.  Since $\calM_{D_1} \cap \calR(1,n-2)$ has codimension $2n-1$ in $\Mat_n(\K)$
and $n>3$, this yields a contradiction.
\end{proof}

Using the definition of $\calR(1,n-2)$, we
deduce that there is a unique $2$-dimensional linear subspace $P_{D_1} \subset E$ such that
$\dim(H_{D_1}x)=1$ for every $x \in P_{D_1} \setminus \{0\}$, whilst $H_{D_1}x=E$ for every $x \in E \setminus P_{D_1}$; moreover
the line
$$D'_1:=H_{D_1}x$$
is independent from $x \in P_{D_1} \setminus \{0\}$
(indeed, given a pair $(P,Q)\in \GL_n(\K)^2$ such that $H_{D_1}=P\,\calR(1,n-2)\,Q$,
one simply has $P_{D_1}=Q^{-1}\Vect(e_{n-1},e_{n-2})$ and $D'_1=P\,\Vect(e_1)$).

\begin{claim}\label{sameplane}
The plane $P_{D_1}$ is independent from the choice of $D_1$, and it contains $e_n$.
\end{claim}

\begin{proof}
Consider the linear subspace $F:=\Vect(e_n)+\underset{D_1 \in \Pgros(\Vect(e_{n-1},e_n))}{\sum}\,P_{D_1}$.
Assume $\dim F \geq 3$ and extend $e_n$ into a linearly independent triple $(e_n,x,y)$ in $F$.
Setting $Y:=g(V')$, we find that $Ye_n=0$, $\dim(Yx) \leq 1$ and $\dim(Yy) \leq 1$, hence
$\codim_{\Mat_n(\K)} Y \geq n+2(n-1)>2n$, which contradicts the fact that $\dim V'=n(n-2)$.
We deduce that $\dim F \leq 2$, which proves that all the planes $P_{D_1}$ are equal and contain $e_n$.
\end{proof}

We may now assume:
\begin{equation}\label{H1}
P_{D_1}=\Vect(e_{n-1},e_n) \quad \text{for every $D_1 \in \Pgros(\Vect(e_{n-1},e_n))$.}
\end{equation}
The situation is indeed unchanged should we replace $g$ with $u_{I_n,Q} \circ g$ for any $Q \in \GL_n(\K)$
such that $Qe_n=e_n$.

\begin{claim}
One has $g(V')=V'$.
\end{claim}

\begin{proof}
Choose two arbitrary distinct lines $D_1$ and $D_2$ in $\Pgros(\Vect(e_{n-1},e_n))$,
and notice that $V'=V \cap \calM_{D_1} \cap \calM_{D_2}$ hence
$g(V')=H_{D_1} \cap H_{D_2}=V'$ since $D'_1 \neq D'_2$.
\end{proof}

\begin{claim}\label{planeclaim}
The sum $\calP$ of all lines $D'_1$, for $D_1$ in $\Pgros(\Vect(e_{n-1},e_n))$, is a $2$-dimensional subspace of $E$.
\end{claim}

\begin{proof}
Set $D_1:=\Vect(e_{n-1})$ and $D_2:=\Vect(e_n)$.
Note again that \eqref{H0} and \eqref{H1} are unchanged should $g$ be replaced by $u_{P,I_n} \circ g$ for an arbitrary $P \in \GL_n(\K)$,
so we lose no generality whatsoever assuming that $D'_1=\Vect(e_1)$ and
$D'_2=\Vect(e_2)$ (recall that $D'_1 \neq D'_2$ since $H_{D_1} \neq H_{D_2}$).
For $(i,j)\in \lcro 1,n\rcro$, denote by $E_{i,j}$ the elementary matrix of $\Mat_n(\K)$
with entry $1$ at the spot $(i,j)$ and zero elsewhere. Then $E_{1,n} \in \calR(1,n-2) \cap \calM_{D_1}$,
hence $g(E_{1,n}) \in H_{D_1} \subset \calR(2,n-2)$. Similarly
$g(E_{1,n-1}) \in H_{D_2} \subset \calR(2,n-2)$. Since $g(V')=V'$, we deduce that
$g$ maps $\calR(1,n-2)$ into $\calR(2,n-2)$.
Let finally $D_3$ be an arbitrary line in $\Pgros(\Vect(e_{n-1},e_n))$.
Some non-trivial linear combination $A$ of $E_{1,n-1}$ and $E_{1,n}$ must then belong to $\calM_{D_3}$.
Note that $A \in \calR(1,n-2) \setminus V'$, which shows that $g(A) \in \calR(2,n-2) \setminus V'$
since $g$ is one-to-one and $g(V')=V'$. On the other hand $g(A) \in H_{D_3}$ hence
$g(A)x \in D'_3$ for any $x \in \Vect(e_{n-1},e_n)$. Since $g(A) \not\in V'$, we may then choose
$x$ such that $g(A)x \neq 0$, which shows that $D'_3 \subset \Vect(e_1,e_2)=D'_1+D'_2$.
This shows $\calP=D'_1+D'_2$ and proves our claim.
\end{proof}

Notice in particular that $g(V)$ contains every rank $1$ matrix with image $D'_1$, for $D_1$ in $\Pgros(\Vect(e_{n-1},e_n))$,
hence it contains any matrix $M \in \Mat_n(\K)$ such that $\im(M) \subset \calP$. \\
As in the beginning of the proof of Claim \ref{planeclaim}, we lose no generality assuming that $\calP=\Vect(e_1,e_2)$
and $D'_1=\Vect(e_1)$ for $D_1:=\Vect(e_n)$.

Note then that, for $D:=D_1=\Vect(e_n)$, one has $g(V \cap \calM_D)=\calR(1,n-2)$
hence $(g^{-1})^{-1}(\calM_D)=\calR(1,n-2)$ and $g^{-1}(\calM_D)=\calR(1,n-2)$, and moreover
$g^{-1}(V')=V'$. Since $g(V)$ contains both $\calR(1,n-2)$ and the space of all matrices $M$ with $\im M \subset \Vect(e_1,e_2)$,
one has $\calR(2,n-2) \subset g(V)$.

Now \textbf{we replace $(g,V)$ with $(g^{-1},g(V))$.}
Notice that \eqref{H0} and \eqref{H1} are preserved, but we now have the additional
fact:
\begin{equation}\label{H2}
V \; \text{contains the linear subspace $\calR(2,n-2)$.}
\end{equation}
Note that the reductions of the present section preserve \eqref{H2} hence we lose no generality assuming that\footnote{\eqref{H3} and \eqref{H4} are respectively \eqref{H0} and the result of Claim 10 for the
``old" pair $(g,V)$, i.e.\ for $(g^{-1},g(V))$ with our new notations.}:
\begin{equation}\label{H3}
g(V \cap \calM_{\Vect(e_n)})=\calR(1,n-2)
\end{equation}
and
\begin{equation}\label{H4}
H_{D_1} \subset \calR(2,n-2) \quad \text{for every line $D_1 \subset \Vect(e_{n-1},e_n)$.}
\end{equation}
Noting that $\calR(2,n-2) \subset \underset{D_1 \in \Pgros(\Vect(e_{n-1},e_n))}{\sum} \calM_{D_1}$,
this yields the inclusion $g\bigl(\calR(2,n-2)\bigr) \subset \calR(2,n-2)$, therefore:
\begin{equation}\label{H5}
g\bigl(\calR(2,n-2)\bigr)=\calR(2,n-2).
\end{equation}

\subsection{Sorting out the action of $g$ on $\calR(2,n-2)$}

Since $g$ stabilizes both $V'$ and $\calR(2,n-2)$, we deduce that there is a linear automorphism
$\varphi$ of $\Mat_2(\K)$ such that, for every $M=\begin{bmatrix}
? & A \\
? & 0
\end{bmatrix}$ with $A\in \Mat_2(\K)$, one has $g(M)=\begin{bmatrix}
? & \varphi(A) \\
? & 0
\end{bmatrix}$.
Since $g$ preserves non-singularity, it follows that $\varphi$ must also preserve non-singularity, hence
Dieudonn\'e's theorem shows that it is a Frobenius automorphism.
However, we know that $g$ maps $\calR(1,n-2)$ into the set of matrices with zero as last column,
hence $\varphi$ may not be some $u_{P,Q}$.
Hence $\varphi=v_{P,Q}$ for some pair $(P,Q)\in \GL_2(\K)^2$.
The initial assumption $g^{-1}(\calM_{\Vect(e_n)})=\calR(1,n-2)$ then yields
$Qe_2 \in \Vect(e_2)$, whilst the intermediate one $g(V \cap \calM_{\Vect(e_n)})=\calR(1,n-2)$
yields $Pe_1\in \Vect(e_1)$. We thus lose no generality assuming:
\begin{equation}\label{H6}
\varphi(A)=A^T \quad \text{for every $A \in \Mat_2(\K)$}
\end{equation}
(indeed, in the general case, replace $g$ with $g':=u_{P',Q'}\circ g$ with
$P':=\begin{bmatrix}
P^{-1} & 0 \\
0 & I_{n-2}
\end{bmatrix}$ and $Q':=
\begin{bmatrix}
I_{n-2} & 0 \\
0 & Q^{-1}
\end{bmatrix}$, and check that $g'$ satisfies assumptions \eqref{H0}, \eqref{H1}, \eqref{H3}, \eqref{H4}, \eqref{H5} and \eqref{H6}). \\
Let $L \in \Mat_{2,n-2}(\K)$ and set $M_L:=\begin{bmatrix}
L & 0 \\
0 & 0
\end{bmatrix}$.
Let $N \in \GL_{n-2}(\K)$, and denote by $M$ the matrix of $\calR(2,n-2)$ such that
$g(M)=\begin{bmatrix}
0 & I_2 \\
N & 0
\end{bmatrix}$.  Then
$g(M_L)=\begin{bmatrix}
? & 0 \\
L' & 0
\end{bmatrix}$ for some $L' \in \Mat_{n-2}(\K)$.
However $M+M_L$ is non-singular, hence
$\begin{bmatrix}
? & I_2 \\
N+L' & 0
\end{bmatrix}$ is non-singular, which shows that $N+L'$ is non-singular.
It follows that $N+L'$ is non-singular for every $N \in \GL_{n-2}(\K)$, and the next lemma shows that $L'=0$.

\begin{lemme}\label{addtononsingular}
Let $A \in \Mat_p(\K)$ be such that $\forall P \in \GL_p(\K), \; A+P \in \GL_p(\K)$. Then $A=0$.
\end{lemme}

\begin{proof}
Using the equivalence of matrices,
we lose no generality assuming that $A=\begin{bmatrix}
I_q & 0 \\
0 & 0
\end{bmatrix}$ for some $q \in \lcro 0,p\rcro$. If $q>0$, then taking $P:=-I_n$ yields a contradiction.
Hence $q=0$ and $A=0$.
\end{proof}

We now deduce that $g$ stabilizes the subspace of all matrices of the form
$\begin{bmatrix}
L & 0 \\
0 & 0
\end{bmatrix}$ with $L \in \Mat_{2,n-2}(\K)$.
Since $g$ also stabilizes $V'$, it follows that there is an automorphism $\psi$ of $\Mat_{n-2}(\K)$ such that
any matrix of the form
$\begin{bmatrix}
? & 0 \\
B & 0
\end{bmatrix}$ with $B \in \Mat_{n-2}(\K)$ is mapped by $g$ to $\begin{bmatrix}
? & 0 \\
\psi(B) & 0
\end{bmatrix}$.

\subsection{The final contradiction}

The final contradiction will now come by considering the structure of the subspace
$H:=g(V \cap \calM^{\Vect(e_1)})$.

\begin{claim}
There is no line $D_1 \subset E$ such that $H \subset \calM^{D_1}$.
\end{claim}

\begin{proof}
Set $H':=\calR(2,n-2) \cap \calM^{\Vect(e_1)}$, i.e.\ $H'$ is the set of matrices of the form
$\begin{bmatrix}
0 & 0 \\
L_1 & L_2 \\
B & 0
\end{bmatrix}$ with $L_1 \in \Mat_{1,n-2}(\K)$, $L_2 \in \Mat_{1,2}(\K)$ and $B \in \Mat_{n-2}(\K)$.
On the one hand, applying $g$ to those matrices with $L_1=0$ and $B=0$
shows that the subspace $g(H')E$ contains $\Vect(e_1,e_2)$ (by \eqref{H6}). On the other hand, $\psi$ is an automorphism hence
applying $g$ to those matrices with $L_1=0$ and $L_2=0$
shows that the projection of $g(H')E$ on $\Vect(e_3,\dots,e_n)$ alongside $\Vect(e_1,e_2)$ is onto.
This shows that $g(H')E=E$, hence $g(H')$ may not be included in any $\calM^{D_1}$, which proves our claim since $H' \subset V \cap \calM^{\Vect(e_1)}$.
\end{proof}

\begin{claim}
There is no line $D_1 \subset E$ such that $H \subset \calM_{D_1}$.
\end{claim}

\begin{proof}
Assume that there is a line $D_1 \in \Pgros(E)$ such that $H \subset \calM_{D_1.}$
\begin{itemize}
\item Assume first that $D_1 \subset \Vect(e_{n-1},e_n)$.
Then, applying Theorem \ref{atklloydtheorem} to $g^{-1} : g(V) \hookrightarrow \Mat_n(\K)$, we would have
$g(V \cap \calM^{\Vect(e_1)})=g(V) \cap \calM_{D_1}$, hence
$g^{-1}(g(V) \cap \calM_{D_1})=V \cap \calM^{\Vect(e_1)}$.
However, $g^{-1}$ satisfies condition \eqref{H0} so we may apply Claim \ref{impossclaim1} to it and obtain a contradiction.
\item We deduce that $D_1 \not\subset \Vect(e_{n-1},e_n)$.
Since $g$ stabilizes $V'$, it follows that every matrix of $g(V' \cap \calM^{\Vect(e_1)})$
vanishes on the $3$-dimensional subspace $D_1+ \Vect(e_{n-1},e_n)$, hence
$\codim g(V' \cap \calM^{\Vect(e_1)}) \geq 3n$, contradicting the fact that $V' \cap \calM^{\Vect(e_1)}$ has codimension $3n-2$.
\end{itemize}
\end{proof}

Applying the previous claims together with Claim \ref{basiclaim2},
we deduce that $H=g(V \cap \calM^{\Vect(e_1)})$ is equivalent to $\calR(1,n-2)$
or to $\calR(n-2,1)$. In any case, $H$ is spanned by its rank $1$ matrices, which will yield
a final contradiction, as we shall see.

Let $M \in H$, and write
$M=\begin{bmatrix}
? & ? \\
? & \alpha(M)
\end{bmatrix}$ with $\alpha(M) \in \Mat_{n-2,2}(\K)$.
Since $g(V)$ contains $\calR(2,n-2)$, we deduce that
$g(V)$ contains the matrix
$\begin{bmatrix}
0 & 0 \\
0 & \alpha(M)
\end{bmatrix}$. Assume $\alpha(M)$ has rank $1$ and let $\begin{bmatrix}
a \\
b
\end{bmatrix}$ be a non-zero vector in its kernel.
Then $g(V)$, which contains $\calR(2,n-2)$, also contains the matrix
$\begin{bmatrix}
0 & A \\
0 & 0
\end{bmatrix}$ for every singular matrix $A \in \Mat_2(\K)$ with kernel $\K\begin{bmatrix}
a \\
b
\end{bmatrix}$.
Setting $D:=\K\begin{bmatrix}
0 & \dots & 0 & a & b
\end{bmatrix}^T$, we deduce that
the intersection of $g(V)$ with the set of all matrices $M$ such that $\Vect(e_1,\dots,e_{n-2}) \oplus D \subset \Ker M$
has a dimension greater than $2$. This however yields a contradiction
because it would show that $\codim(g(V) \cap \calM_D)<2n-2$, whereas
Claim \ref{structureclaim} applies to $g^{-1}$ and shows that $\codim g(V) \cap \calM_D=2n-2$.

We deduce that if $\rk M=1$, then $\alpha(M)=0$.
Since $\alpha$ is linear and $H$ is spanned by its rank $1$ matrices, we deduce that $\alpha=0$.
This shows that $H \subset \calR(2,n-2)$. However $g(\calR(2,n-2))=\calR(2,n-2)$ and $g$ is one-to-one hence
$V \cap \calM^{\Vect(e_1)} \subset \calR(2,n-2)$.
It follows that $V$ contains no matrix of the form
$\begin{bmatrix}
0 & 0 \\
0 & C
\end{bmatrix}$ for some $C \in \Mat_{n-2,2}(\K) \setminus \{0\}$ hence
$\codim V \geq 2(n-2)>n-2$, a final contradiction.
Thus Lemma \ref{hightechlemma} is proven at last, which completes the proof of Theorem \ref{strong}.

\section{The weak preservers of non-singular matrices}\label{weakproof}

In this section, we turn to the proof of Theorems \ref{weakgeneral} and \ref{weakKinfinite}.
First, notice that Theorem \ref{weakgeneral} trivially derives from Theorem \ref{strong} when $\K$ is finite:
indeed, in this case, if $f : V \rightarrow V$ is one-to-one and stabilizes $V \cap \GL_n(\K)$,
then we have $f^{-1}(\GL_n(\K))=V \cap \GL_n(\K)$ since $V \cap \GL_n(\K)$ is finite.
In the case $\K$ is infinite, Theorem \ref{weakgeneral} will be deduced from Theorem \ref{weakKinfinite}.

We now try to derive Theorem \ref{weakKinfinite} from Theorem \ref{strong}.
It obviously suffices to prove the following proposition:

\begin{prop}\label{weakimpliesstrong}
Assume $\K$ is infinite.
Let $V$ be a linear subspace of $\Mat_n(\K)$ such that $\codim V<n-1$, and $f : V \hookrightarrow \Mat_n(\K)$
be a linear embedding such that $f(V \cap \GL_n(\K)) \subset \GL_n(\K)$. Then $f^{-1}(\GL_n(\K))=V \cap \GL_n(\K)$.
\end{prop}

In order to show this, we will generalize a method of \cite{Botta} by
considering polynomial functions over the $\K$-vector space $V$.
Since $\K$ is infinite, these can be treated as algebraic polynomials.
Notice in particular that if $V$ is a linear subspace of $\Mat_n(\K)$, then
$\det_{|V}$, the restriction of the determinant to $V$, is a homogeneous polynomial of degree $n$.

In order to establish Proposition \ref{weakimpliesstrong}, we successively prove the following two results:

\begin{prop}\label{detirr}
Assume $\K$ is infinite.
Let $V$ be a linear subspace of $\Mat_n(\K)$ such that $\codim V \leq \max(n-2,0)$.
Then $\det_{|V}$ is irreducible.
\end{prop}

\begin{prop}\label{detdivides}
Assume $\K$ is infinite.
Let $V$ be a linear subspace of $\Mat_n(\K)$ such that $\codim V \leq \max(n-2,0)$,
and $p : V \rightarrow \K$ be a polynomial function such that $p(M)=0$ whenever $M\in V$ is singular.
Then $p$ is a multiple of $\det_{|V}$.
\end{prop}

Before proving those results, let us see right away how they may help us prove Proposition \ref{weakimpliesstrong}:

\begin{proof}[Proof of Proposition \ref{weakimpliesstrong}]
Consider the polynomial function $p :=\det_{|V} \circ f^{-1}$ on $f(V)$. Then $p$ is homogeneous with degree
$n$. The assumptions on $f$ show that $f^{-1}(M)$ is singular whenever $M \in f(V)$ is singular,
hence Proposition \ref{detdivides} applied to $f(V)$ shows that $p$ is a multiple of $\det_{|f(V)}$.
However, since $\det_{|f(V)}$ also has degree $n$, we deduce that $p=\lambda\,\det_{|f(V)}$
for some $\lambda \in \K$.
This yields $\det(M)=\lambda\,\det(f(M))$ for every $M \in V$. Since $\det_{|V}$ is irreducible, it is non-zero hence $\lambda \neq 0$.
This shows that $f^{-1}(\GL_n(\K))=V \cap \GL_n(\K)$.
\end{proof}

In order to prove Propositions \ref{detirr} and \ref{detdivides}, we first reduce the situation
to a more elementary one. For $M \in \Mat_n(\K)$, write
$M=\begin{bmatrix}
? & ? \\
? & K(M)
\end{bmatrix}$ with $K(M) \in \Mat_{n-1}(\K)$. \\
Given a linear subspace $V$ of $\Mat_n(\K)$, we denote by $V'$ the linear subspace of matrices of $V$ with a zero first column.
For $(i,j)\in \lcro 1,n\rcro^2$, denote by $E_{i,j}$ the elementary matrix with entry $1$ at the spot $(i,j)$
and $0$ elsewhere. Assume that $\Vect(E_{1,2},\dots,E_{1,n})\not\subset V$.
Then the rank theorem shows that
$$\codim_{\Mat_{n-1}(\K)} K(V') \leq \codim_{\Mat_n(\K)}V-1$$
therefore
$$\codim_{\Mat_{n-1}(\K)} K(V) \leq \codim_{\Mat_n(\K)}V-1.$$
We may now state the basic lemma that we will use:

\begin{lemme}\label{lemme13}
Let $V$ be a linear subspace of $\Mat_n(\K)$ such that $\codim V \leq \max(n-2,0)$.
Then $V$ is equivalent to a linear subspace $W$ which contains $E_{1,1}$ and for which
$\codim K(W') \leq \max(n-3,0)$.
\end{lemme}

\begin{proof}
The result is trivial when $\codim V=0$. Assume $\codim V>0$.
Then there must be an index $i \in \lcro 1,n\rcro$ such that $\Vect(E_{i,1},\dots,E_{i,n}) \not\subset V$.
Using row operations, we lose no generality assuming $i=1$. However, since $\codim V<n$,
we have $\Vect(E_{1,1},\dots,E_{1,n}) \cap V \neq \{0\}$. Using a series of column operations, we may then assume
furthermore that $E_{1,1} \in V$, whereas $\Vect(E_{1,2},\dots,E_{1,n}) \not\subset V$. With the above inequalities,
this leads to $\codim K(V') \leq n-3$.
\end{proof}

\begin{proof}[Proof of Proposition \ref{detirr}]
We use an induction on $n$. The result is trivial when $n=1$. Set an arbitrary integer $n>0$ and
assume that the result holds for $n-1$. Let $V \subset \Mat_n(\K)$ be a linear subspace such that
$\codim V \leq \max(n-2,0)$.
Notice that the problem is essentially unchanged should $V$ be replaced with $u(V)$
for some Frobenius automorphism. By Lemma \ref{lemme13}, we lose no generality assuming that
$V$ contains $E_{1,1}$ and $\codim K(V') \leq \max(n-2,0)$.
Assume $\det_{|V}=p\,q$ for some non-constant polynomial functions $p$ and $q$.
For any $M\in V'$, we then have
$$p(E_{1,1}+M)\,q(E_{1,1}+M)=\det K(M).$$
However, the induction hypothesis shows that the homogeneous polynomial function $\det_{|K(V')}$ is irreducible, hence
the homogeneous polynomial function $M \mapsto \det K(M)$ on $V'$ also is.
We then lose no generality assuming that $M \mapsto p(E_{1,1}+M)$ is a non-zero scalar multiple of
$M \mapsto \det K(M)$, hence has total degree $n-1$. Since $\det_{|V}$ is homogeneous,
$p$ and $q$ are also homogeneous hence $q$ must have degree $1$, i.e.\ $q$ is a linear form.
It follows that every matrix of $\Ker q$ is singular. However, $\codim_{\Mat_n(\K)} \Ker q \leq \codim_{\Mat_n(\K)} V+1
\leq n-1$ hence the Dieudonn\'e theorem \cite{Dieudonne} shows that $\Ker q$ must contain a non-singular matrix. This is a contradiction,
which shows that $\det_{|V}$ is irreducible.
\end{proof}

\begin{proof}[Proof of Proposition \ref{detdivides}]
As in the proof of Proposition \ref{detirr}, we lose no generality
assuming that the linear subspace $V$ contains $E_{1,1}$ and that $\codim K(V) \leq \max(n-2,0)$.
Define now $V'':=\bigl\{M \in V : \; m_{1,1}=0\bigr\}$ so that $V=V'' \oplus \K E_{1,1}$ and $K(V)=K(V'')$.
Development of the determinant along the first column yields a polynomial function $q : V'' \rightarrow \K$ such that
$$\forall (x,M)\in \K \times V'', \; \det(xE_{1,1}+M)=x\,\det K(M)+q(M).$$
Using the Euclidian algorithm with respect to the indeterminate $x$,
we may then find two polynomial functions $r : \K \times V'' \rightarrow \K$
and $s : V''\rightarrow \K$, together with a positive integer $N$ such that
$$\forall (x,M)\in \K \times V'', \; (\det K(M))^N\, p(xE_{1,1}+M)=\det(xE_{1,1}+M)\,r(x,M)+s(M)$$
and we may even assume that $s$ is a multiple of the polynomial function $M \mapsto \det K(M)$ (on $V''$).
Let $M \in V''$ such that $\det K(M) \neq 0$. Then we may find some $x \in \K$ such that $\det(xE_{1,1}+M)=0$, hence
$p(xE_{1,1}+M)=0$ and we deduce that $s(M)=0$. \\
This shows that $s=0$, hence $\det_{|V}$ divides the polynomial function $M \mapsto (\det K(M))^N\, p(M)$ on $V$.
However, we know from Proposition \ref{detirr} that both $\det_{|V}$ and $M \mapsto \det K(M)$
are irreducible homogeneous polynomial functions on $V$, with respective degrees $n$ and $n-1$. Therefore
$\det_{|V}$ may not divide the latter, which shows that it divides $p$.
\end{proof}

\section{The exceptional case of linear hyperplanes of $\Mat_3(\F_2)$}\label{F2}

\subsection{Reduction to the case of an internal linear preserver}\label{redtointernalF2}

In this section, we wish to examine more closely the situation of linear hyperplanes of $\Mat_3(\F_2)$.
The major obstruction for proving Theorems \ref{strong} and \ref{weakgeneral} in this case
is the counterexample in the Atkinson-Lloyd theorem.
Recall from Theorem 2 of \cite{dSPclass} that every $5$-dimensional singular linear subspace $V$
of $\Mat_3(\F_2)$ satisfies one of the mutually exclusive conditions:
\begin{enumerate}[(i)]
\item $V \subset \calM^D$ for a (unique) line $D \subset \F_2^3$;
\item $V \subset \calM_D$ for a (unique) line $D \subset \F_2^3$;
\item $V$ is equivalent to $\calR(1,1)$;
\item $V$ is equivalent to the subspace
$$\calJ_3(\F_2):=\Biggl\{
\begin{bmatrix}
a & 0 & 0 \\
c & b & 0 \\
d & e & a+b
\end{bmatrix} \mid (a,b,c,d,e)\in \F_2^5\Biggr\}$$
i.e.\ to the subspace of lower triangular matrices with trace zero.
\end{enumerate}
This last case is one major obstacle both in the proof of Lemma \ref{hightechlemma}
and in that of Claim \ref{inverseimageclaim}.
Notice however that if the result of Claim \ref{inverseimageclaim} holds for some
linear embedding $f : V \hookrightarrow \Mat_3(\F_2)$ of a hyperplane $V$ such that $f$ strongly
preserves non-singularity, then the rest of the proof from Section \ref{strongstart}
applies and shows that $f$ extends to a Frobenius automorphism of $\Mat_3(\F_2)$.

\vskip 3mm
We reduce the study to three cases.
Using the non-degenerate symmetric bilinear form $(A,B) \mapsto \tr(AB)$ on $\Mat_n(\K)$, we see that
orbits of hyperplanes of $\Mat_n(\K)$ are classified by the orbits of their orthogonal subspace (which is always a line), i.e.\
by the rank of the non-zero matrices in their orthogonal subspace. It follows that there are exactly
$n$ orbits of hyperplanes of $\Mat_n(\K)$ under equivalence, and in the case at hand,
every hyperplane of $\Mat_3(\F_2)$ is equivalent to one and only one of the three particular hyperplanes:
\begin{enumerate}[(a)]
\item $\calV_1(\F_2):=
\biggl\{
\begin{bmatrix}
M & C \\
L & 0
\end{bmatrix} \mid (M,C,L)\in \Mat_2(\F_2) \times \Mat_{2,1}(\F_2) \times \Mat_{1,2}(\F_2)
\biggr\}$;
\item $\calV_2(\F_2):=
\biggl\{
\begin{bmatrix}
M & C \\
L & a
\end{bmatrix} \mid (M,C,L,a)\in \frak{sl}_2(\F_2) \times \Mat_{2,1}(\F_2) \times \Mat_{1,2}(\F_2) \times \F_2
\biggr\}$;
\item $\frak{sl}_3(\F_2)$.
\end{enumerate}

Therefore, we lose no generality assuming that $V$ is one of these three hyperplanes,
and we will actually study the three cases separately.
In order to do this, it will be convenient to reduce the situation to the case where $f(V)=V$.
This is done thanks to the next result:

\begin{prop}\label{linpresimpliqueequivalent}
Let $V$ and $V'$ be linear hyperplanes of $\Mat_3(\F_2)$ and $f : V \rightarrow V'$
be a linear isomorphism such that $f^{-1}(\GL_3(\F_2))=V \cap \GL_3(\F_2)$.
Then $V$ and $V'$ are equivalent.
\end{prop}

The case one of the hyperplanes $V$ and $V'$ is equivalent to $\calV_1(\F_2)$ is easy:
indeed, $\calV_1(\F_2)$ contains a $6$-dimensional singular subspace.
However, if a linear hyperplane $V''$ of $\Mat_3(\F_2)$ contains such a subspace, then the Dieudonn\'e theorem
on singular subspaces shows that $\calM^D \subset V''$ or $\calM_D \subset V''$ for some line $D \subset \F_2^3$,
and this proves that every matrix of $(V'')^\bot$ has a rank lesser than or equal to $1$, hence
$V''$ is equivalent to $\calV_1(\F_2)$.
We deduce that if $V$ or $V'$ is equivalent to $\calV_1(\F_2)$, then so is the other one.

The remaining cases rely upon a counting argument: we show that
$\#\bigl(\calV_2(\F_2) \cap \GL_3(\F_2)\bigr) \neq \#\bigl(\frak{sl}_3(\F_2) \cap \GL_3(\F_2)\bigr)$, which
clearly yields Proposition \ref{linpresimpliqueequivalent}.

\begin{prop}
The space $\frak{sl}_3(\F_2)$ has $80$ non-singular elements.
\end{prop}

\begin{proof}
A matrix $M \in \Mat_3(\F_2)$ belongs to $\frak{sl}_3(\F_2) \cap \GL_3(\F_2)$ if and only if its
characteristic polynomial is $t^3+t+1$ or $t^3+1$.
Recall that $\# \GL_3(\F_2)=(2^3-1)(2^3-2)(2^3-2^2)=7 \times 6 \times 4$.
\begin{itemize}
\item Note that $t^3+t+1$ is irreducible in $\F_2[t]$ hence the matrices of $\Mat_3(\F_2)$ with characteristic polynomial
$t^3+t+1$ form a single orbit under similarity, and the companion matrix of $t^3+t+1$ is one of them.
Moreover, the centralizer of this companion matrix in the algebra $\Mat_3(\F_2)$ is
$\F_2[t]/(t^3+t+1) \simeq \F_8$ since $t^3+t+1$ is irreducible: therefore this centralizer contains exactly $7$ non-singular matrices.
It follows that there are $6 \times 4=24$ matrices of $\Mat_3(\F_2)$ with characteristic polynomial $t^3+t+1$.
\item We may factorize $t^3+1=(t+1)(t^2+t+1)$. If a matrix has $t^3+1$ as characteristic polynomial,
then it must also have $t^3+1$ as minimal polynomial hence it is similar both to the companion matrix of $t^3+1$
and to the matrix $A=\begin{bmatrix}
1 & 0 & 0 \\
0 & 0 & 1 \\
0 & 1 & 1
\end{bmatrix}$. Any matrix that commutes with $A$ must stabilize $\Ker(A-I_3)$ and $\im(A-I_3)$, therefore the centralizer of $A$ in
$\Mat_3(\F_2)$ is the set of matrices of the form
$\begin{bmatrix}
a & 0 \\
0 & B
\end{bmatrix}$ where $a \in \F_2$ and $B \in \Mat_2(\F_2)$ commutes with $\begin{bmatrix}
0 & 1 \\
1 & 1
\end{bmatrix}$. For such a matrix to be non-singular, it is necessary and sufficient that
$a=1$ and $B$ be non-singular, which leaves $3$ possibilities (notice that the centralizer of
the companion matrix $\begin{bmatrix}
0 & 1 \\
1 & 1
\end{bmatrix}$ in $\Mat_2(\F_2)$ is isomorphic to $\F_2[t]/(t^2+t+1) \simeq \F_4$).
We conclude that there are $7 \times 2 \times 4=56$ matrices in $\Mat_3(\F_2)$ with characteristic polynomial $t^3+1$.
\end{itemize}
\end{proof}

\begin{prop}
The space $\calV_2(\F_2)$ has 88 non-singular elements.
\end{prop}

\begin{proof}
Let $M \in \frak{sl}_2(\F_2)$. We count the triples $(L,C,x)\in \Mat_{1,2}(\F_2) \times \Mat_{2,1}(\F_2) \times \F_2$
such that $\begin{bmatrix}
M & C \\
L & x
\end{bmatrix}$ is non-singular. If $M=0$, then there is no such triple. Assume $M$ is non-singular. Then
the former matrix has determinant $L\widetilde{M}C-x$, where $\widetilde{M}$ denotes the transpose of the matrix of cofactors of $M$.
Hence there are $2^4$ well-suited triples (we choose $L$ and $C$ freely, and then $x$ accordingly).
Since $\frak{sl}_2(\F_2)$ contains exactly $4$ non-singular matrices, we find $2^6$ non-singular matrices in $\calV_2(\F_2)$
of the former type. \\
Assume finally that $\rk M=1$. Then $M$ is nilpotent and we thus lose no generality assuming that $M=\begin{bmatrix}
0 & 1 \\
0 & 0
\end{bmatrix}$. However, given a $5$-tuple $(a,b,c,d,e)\in \K^5$, one has
$\begin{vmatrix}
0 & 1 & a \\
0 & 0 & b \\
c & d & e
\end{vmatrix}=bc$, hence there are $2^3$ well-suited triples $(L,C,x)$ for $M$.
Since $\frak{sl}_3(\F_2)$ contains exactly three rank $1$ matrices, we find that
$\calV_2(\F_2)$ contains exactly $2^6+3\times 2^3=88$ non-singular matrices.
\end{proof}

\subsection{The case of $\calV_1(\F_2)$}

Here, we prove the following result:

\begin{prop}
There exists a linear automorphism $f$ of $\calV_1(\F_2)$ which (strongly) preserves
non-singularity but does not extend to a Frobenius automorphism of $\Mat_3(\F_2)$.
\end{prop}

\begin{proof}
Let $\alpha : \Mat_{1,2}(\F_2) \rightarrow \Mat_2(\F_2)$ and $\beta : \Mat_{2,1}(\F_2) \rightarrow \Mat_2(\F_2)$
be arbitrary linear maps. We will show that we may choose $\alpha$ and $\beta$ so that the linear automorphism
$$f : \begin{bmatrix}
M & C \\
L & 0
\end{bmatrix} \longmapsto \begin{bmatrix}
M+\alpha(L)+\beta(C) & C \\
L & 0
\end{bmatrix}$$
has the claimed properties.
\begin{itemize}
\item In order to do this, we first study on what conditions on $\alpha$ and $\beta$
the map $f$ may be extended to a Frobenius automorphism.
A sufficient condition is easy to find: if $\alpha : L \mapsto \begin{bmatrix}
a\,L \\
b\,L
\end{bmatrix}$ and $\beta : C \mapsto \begin{bmatrix}
c\,C & d\,C
\end{bmatrix}$ for some $(a,b,c,d)\in \F_2^4$, then $f(M)$ is simply obtained from $M$ by performing a series of row and column operations
(that is independent from $M$), hence $f$ clearly extends to a Frobenius automorphism. \\
Conversely, assume that $f=u_{P,Q}$ or $f=v_{P,Q}$ for some $(P,Q)\in \GL_3(\F_2)^2$. Notice for every $M \in \Mat_2(\F_2)$
that $u_{P,Q}$ fixes the matrix
$\begin{bmatrix}
M & 0 \\
0 & 0
\end{bmatrix}$ or maps it to its transpose, i.e.\ $u_{P,Q}$ fixes or transposes
every matrix with image $\Vect(e_1,e_2)$
and kernel $\Vect(e_3)$, where $(e_1,e_2,e_3)$ is the canonical basis of $\F_2^3$.
It easily follows that $P$ stabilizes $\Vect(e_1,e_2)$ and $Q$ stabilizes $\Vect(e_3)$, hence there are matrices
$C_1 \in \Mat_{2,1}(\F_2)$, $L_1 \in \Mat_{1,2}(\F_2)$ and matrices $P_1$ and $Q_1$ in $\GL_2(\F_2)$ such that
$P=\begin{bmatrix}
P_1 & C_1 \\
0 & 1
\end{bmatrix}$ and
$Q=\begin{bmatrix}
Q_1 & 0 \\
L_1 & 1
\end{bmatrix}$. Computing the image by $f$ of the previous matrices shows that
$\forall M \in \Mat_2(\F_2), \; P_1MQ_1=M$
or $\forall M \in \Mat_2(\F_2), \; P_1MQ_1=M^T$. In any case, taking $M=I_2$ yields $Q_1=P_1^{-1}$. \\
In the first case, $P_1$ commutes with every matrix of $\Mat_2(\F_2)$, which shows that $P_1=I_2=Q_1$, and we then notice
that $\alpha$ and $\beta$ have the aforementioned form. \\
However, the second case leads to a contradiction by taking every $M$ with zero as second column.
\end{itemize}
We now prove that $\alpha$ and $\beta$ may be chosen so that $f$ is not
a Frobenius automorphism although it is a determinant preserver.
Let $\begin{bmatrix}
M & C \\
L & 0
\end{bmatrix} \in \calV_1(\F_2)$. Its determinant is $L\,\widetilde{M}\,C$ (recall that $\widetilde{M}$ denotes the transpose of
the matrix of cofactors of $M$).
However, $M \mapsto \widetilde{M}$ is linear.
It follows that $f$ is a determinant preserver if (and only if)
$$\forall (L,C)\in \Mat_{1,2}(\F_2) \times \Mat_{2,1}(\F_2), \; L(\widetilde{\alpha(L)}+\widetilde{\beta(C)})\,C=0.$$
Taking
$$\alpha : \; \begin{bmatrix}
l_1 & l_2
\end{bmatrix} \mapsto \begin{bmatrix}
0 & 0 \\
l_2 & 0
\end{bmatrix} \quad \text{and} \quad \beta : \; \begin{bmatrix}
c_1 \\
c_2
\end{bmatrix} \mapsto \begin{bmatrix}
0 & 0 \\
c_1 & 0
\end{bmatrix},$$
it is obvious from the above necessary condition that $f$ is not a Frobenius automorphism,
However, for every $L=\begin{bmatrix}
l_1 & l_2
\end{bmatrix} \in \Mat_{1,2}(\F_2)$ and every
$C=\begin{bmatrix}
c_1 \\
c_2
\end{bmatrix} \in \Mat_{2,1}(\F_2)$, one has
$$L(\widetilde{\alpha(L)}+\widetilde{\beta(C)})\,C=l_2(l_2+c_1)c_1)=l_2^2c_1+l_2c_1^2=2 l_2c_1=0,$$
and therefore $f$ is a determinant preserver.
\end{proof}

\subsection{The case of $\calV_2(\F_2)$}

Here, we let $f : \calV_2(\F_2) \rightarrow \calV_2(\F_2)$ be a linear transformation which preserves non-singularity.
We wish to prove that Claim \ref{inverseimageclaim} holds in this situation.
We do this by analyzing the $5$-dimensional singular subspaces of $\calV_2(\F_2)$.
Recall from section \ref{redtointernalF2} (or prove this elementary fact directly) that $\calV_2(\F_2) \cap \calM_D$ and $\calV_2(\F_2) \cap \calM^D$
have dimension $5$ for any line $D \subset \F_2^3$.
In order to simplify the discourse, we will say that a $5$-dimensional singular subspace $V$ of $\Mat_3(\F_2)$ is:
\begin{itemize}
\item \textbf{maximal of the first kind} if equivalent to $\calR(1,1)$;
\item \textbf{maximal of the second kind} if equivalent to $\calJ_3(\F_2)$;
\item \textbf{non-maximal} if $V\subset \calM_D$ of $V\subset \calM^D$ for some line $D \subset \F_2^3$.
\end{itemize}
This terminology stems from the problem of maximality in the set of singular linear subspaces of $\Mat_3(\F_2)$ ordered by the inclusion of subsets.
Since $\F_2^3$ has $7$ non-zero vectors, and therefore $7$ one-dimensional subspaces,
$\calV_2(\F_2)$ has exactly fourteen non-maximal $5$-dimensional singular subspaces.
Let us now consider the maximal ones.

\begin{claim}\label{classfirsttype}
The linear subspace
$$\calF:=\Biggl\{
\begin{bmatrix}
0 & 0 & a \\
0 & 0 & b \\
c & d & e
\end{bmatrix} \mid (a,b,c,d,e)\in \F_2^5\Biggr\}$$
is the sole $5$-dimensional maximal singular subspace of the first kind in $\calV_2(\F_2)$.
\end{claim}

\begin{proof}
Clearly, $\calF$ is equivalent to $\calR(1,1)$ and is included in $\calV_2(\F_2)$.
Conversely, set $J_2:=\begin{bmatrix}
I_2 & 0 \\
0 & 0
\end{bmatrix} \in \Mat_3(\F_2)$. Let $V$ be a $5$-dimensional maximal singular subspace in
$\calV_2(\F_2)$ of the first kind. Then there are two non-zero vectors $X_1$ and $X_2$ in $\F_2^3$
such that $\calV$ contains $X_1Y^T$ and $YX_2^T$ for every $Y \in \F_2^3$ and
$\calV$ is actually spanned by those matrices.
Writing that those matrices are orthogonal to $J_2$ for the symmetric bilinear form $(A,B) \mapsto \tr(AB)$, we find that
$J_2X_1=0$ and $X_2^TJ_2=0$, which shows that $X_1$ and $X_2$ are both scalar multiples of $\begin{bmatrix}
0 \\
0 \\
1
\end{bmatrix}$. This shows that $V=\calF$.
\end{proof}

\begin{claim}\label{classsecondtype}
There are exactly three $5$-dimensional maximal singular subspaces of the second kind in $\calV_2(\F_2)$.
One of them is
$$\calG:=\Biggl\{
\begin{bmatrix}
0 & a & c \\
0 & 0 & b \\
a+b & d & e
\end{bmatrix} \mid (a,b,c,d,e)\in \F_2^5\Biggr\}$$
and the two other ones may be obtained by conjugating $\calG$ with $\begin{bmatrix}
P & 0 \\
0 & 1
\end{bmatrix}$ for some $P \in \GL_2(\F_2)$.
\end{claim}

\begin{proof}
Obviously, $\calG$ is equivalent to $\calJ_3(\F_2)$ and is a linear subspace of $\calV_2(\F_2)$. \\
Also, the number $p$ of $5$-dimensional maximal singular subspaces of the second kind
in an hyperplane $V$ which is equivalent to $\calV_2(\F_2)$ is independent from the given $V$.
We now resort to a counting argument. Notice that the orthogonal subspace of
$\calJ_3(\F_2)$ (for $(A,B) \mapsto \tr(AB)$) is the subspace
$$\Biggl\{
\begin{bmatrix}
a & 0 & 0 \\
b & a & 0 \\
d & c & a
\end{bmatrix} \mid (a,b,c,d)\in \F_2^4\Biggr\}:$$
it contains exactly two rank $2$ matrices, hence
$\calJ_3(\F_2)$ is contained in exactly two linear hyperplanes that are equivalent to $\calV_2(\F_2)$
(being equivalent to $\calV_2(\F_2)$ being the same, for a hyperplane of $\Mat_3(\F_2)$, as being orthogonal to a rank $2$ matrix).
It follows from a standard counting argument that $p\,n_1=2\,n_2$,
where $n_1$ denotes the number of hyperplanes in $\Mat_3(\F_2)$ which are equivalent to $\calV_2(\F_2)$,
and $n_2$ the number of $5$-dimensional maximal singular subspaces of the second kind in $\Mat_3(\F_2)$.
\begin{itemize}
\item Clearly, $n_1$ is the number of rank $2$ matrices of $\Mat_3(\F_2)$, hence
$n_1=7\times 7 \times 6$ (there are $7$ possibilities for the kernel of such a matrix and $7 \times 6$ ones for a
linearly independent $2$-tuple in $\F_2^3$).
\item Denote by $(e_1,e_2,e_3)$ the canonical basis of $\F_2^3$.
Let $X \in \F_2^3 \setminus \{0\}$ . Set $\calR_X:=\bigl\{YX^T \mid Y \in \F_2^3\bigr\}$ and
$\calR^X:=\bigl\{XY^T \mid Y \in \F_2^3\bigr\}$.
Notice then that $\calR_X \cap \calJ_3(\F_2)$ has dimension $0$ if $X \in \F_2^3 \setminus \Vect(e_1,e_2)$,
dimension $1$ if $X \in \Vect(e_1,e_2) \setminus \Vect(e_1)$, and dimension $2$ if $X \in \Vect(e_1)$.
Similarly, $\calR^X \cap \calJ_3(\F_2)$ has dimension $0$ if $X \in \F_2^3 \setminus \Vect(e_2,e_3)$,
dimension $1$ if $X \in \Vect(e_2,e_3) \setminus \Vect(e_3)$, and dimension $2$ if $X \in \Vect(e_3)$. \\
This shows that if some pair $(P,Q)\in \GL_3(\F_2)^2$ satisfies
$P\,\calJ_3(\F_2)\,Q^{-1}=\calJ_3(\F_2)$, then $P$ must stabilize $\Vect(e_3)$ and $\Vect(e_2,e_3)$,
whilst $Q^T$ must stabilize $\Vect(e_1)$ and $\Vect(e_1,e_2)$, hence both $P$ and $Q$
are lower triangular. Conversely $P\,\calJ_3(\F_2)\,Q^{-1}=\calJ_3(\F_2)$ for every pair $(P,Q)$ of lower triangular matrices in $\GL_3(\F_2)$.
Since there are $8^2$ such pairs and $\# \GL_3(\F_2)=7 \times 6 \times 4$, we deduce that
$$n_2=\frac{7^2 \times 6^2 \times 4^2}{8^2}=7^2 \times 3^2.$$
\end{itemize}
The previous formulae then yield $p=3$. \\
Let finally $M$ be a non-zero nilpotent matrix of $\Mat_2(\F_2)$.
Set $B:=\begin{bmatrix}
0 & 1 \\
0 & 0
\end{bmatrix}$. Let $P \in \GL_2(\F_2)$ be such that $PBP^{-1}=M$ and set $Q=\begin{bmatrix}
P & 0 \\
0 & 1
\end{bmatrix}$. Then $Q\calG Q^{-1}$ is a $5$-dimensional maximal singular subspace of $\calV_2(\F_2)$ of the second kind and
the projection of $\calG$ onto the first $2 \times 2$ block is $\Vect(M)$.
Since $3$ distinct lines of $\frak{sl}_2(\F_2)$ may be obtained in this manner (there are three non-zero nilpotent matrices in
$\frak{sl}_2(\F_2)$), we deduce that this yields three $5$-dimensional maximal singular subspaces of $\calV_2(\F_2)$ of the second kind,
hence we have found them all.
\end{proof}

In the rest of the proof, we denote by $(e_1,e_2,e_3)$ the canonical basis of $\F_2^3$.

\begin{claim}\label{twokinds}
Let $V$ be a $5$-dimensional singular subspace of $\calV_2(\F_2)$. Then:
\begin{enumerate}[(a)]
\item Either $V \subset \calM_D$ or $V \subset \calM^D$ for some line $D \subset \F_2^3$ not included in $\Vect(e_1,e_2)$;
then $\dim(V \cap V') \leq 3$ for every other $5$-dimensional singular subspace $V'$ of $\calV_2(\F_2)$;
\item Or there exists a $5$-dimensional singular subspace $V'$ of $\calV_2(\F_2)$ such that
$\dim(V \cap V')=4$.
\end{enumerate}
\end{claim}

\begin{proof}
In this proof, we will simply write $\calV_2$ instead of $\calV_2(\F_2)$ to lighten the burden of notations. \\
Assume first that $V \subset \calM_D$ or $V \subset \calM^D$ for some line $D \subset \F_2^3$ not included in $\Vect(e_1,e_2)$.
By transposing, we lose no generality assuming that $V \subset \calM_D$.
We also lose no generality assuming that $D=\Vect(e_3)$.
\begin{itemize}
\item Let $D' \subset \Vect(e_1,e_2,e_3)$ be an arbitrary line distinct from $D$.
A straightforward computation shows that $\dim(\calV_2 \cap \calM_D \cap \calM_{D'})=2$
(notice that we lose no generality assuming $D+D'=\Vect(e_2,e_3)$ for this computation).
\item Let $D' \subset \Vect(e_1,e_2,e_3)$ be an arbitrary line.
Write every matrix $M$ of $\Mat_3(\F_2)$ as
$\begin{bmatrix}
G(M) & ?
\end{bmatrix}$ with $G(M) \in \Mat_{3,2}(\F_2)$.
On the one hand $G(\calV_2 \cap \calM_D)=G(\calV_2)$ is a hyperplane of $\Mat_{3,2}(\F_2)$
and its orthogonal subspace for $b : (A,B) \mapsto \tr(A^TB)$ is $\Vect \begin{bmatrix}
I_2 \\
0
\end{bmatrix}$. On the other hand, the orthogonal subspace of $G(\calM^{D'})$ contains no rank $2$ matrix,
hence $G(\calM^{D'}) \not\subset G(\calV_2)$, and it follows that
$G(\calV_2 \cap \calM_D \cap \calM^{D'})=G(\calV_2) \cap G(\calM^{D'})$ is a hyperplane of $G(\calM^{D'})$
hence
$$\dim (\calV_2 \cap \calM_D \cap \calM^{D'})=\dim G(\calV_2 \cap \calM_D \cap \calM^{D'})=4-1=3.$$
\item Obviously $\dim(V \cap \calF)=2$
and $\dim(V \cap \calG)=2$. Claims \ref{classfirsttype} and \ref{classsecondtype} then entail that $\dim(V \cap V')=2$
for every $5$-dimensional maximal singular subspace $V'$ of $\calV_2$.
\end{itemize}
Assume now that $V=\calV_2 \cap \calM_D$ for some line $D \subset \Vect(e_1,e_2)$.
Then we lose no generality assuming that $D=\Vect(e_1)$. In this case, we have
$\dim(V \cap \calF)=4$. The same obviously holds when $V=\calV_2(\F_2) \cap \calM^D$ for some line $D \subset \Vect(e_1,e_2)$. \\
Assume finally that $V=\calF$ or $V=\calG$. Then clearly $\dim(V \cap \calM_{D_1})=4$ for $D_1=\Vect(e_1)$.
Using Claim \ref{classsecondtype}, this finishes the proof of Claim \ref{twokinds}.
\end{proof}

In the course of the above proof, we have also obtained the following result:

\begin{claim}\label{stabX}
Let $D$ be a line included in $\F_2^3$ but not in $\Vect(e_1,e_2)$.
Set $V:=\calV_2(\F_2) \cap \calM_D$ (resp.\ $V:=\calV_2(\F_2) \cap \calM^D$)
and let $V'$ be a $5$-dimensional singular subspace of $\calV_2(\F_2)$.
Then $\dim(V \cap V')=3$ if and only if $V'=\calV_2(\F_2) \cap \calM^{D'}$
for some line $D' \subset \F_2^3$ (resp.\ $V'=\calV_2(\F_2) \cap \calM_{D'}$
for some line $D' \subset \F_2^3$).
\end{claim}

Recall now that $f : \calV_2(\F_2) \rightarrow \calV_2(\F_2)$ is a linear bijection which (strongly) preserves
non-singularity. Then $f$ permutes the $5$-dimensional singular subspaces of $\calV_2(\F_2)$.
Set
$$\calX:=\bigl\{\calV_2(\F_2) \cap \calM_{\Vect(x)}\mid x \in \F_2^3 \setminus \Vect(e_1,e_2)\bigr\}\,
\bigcup\, \bigl\{\calV_2(\F_2) \cap \calM^{\Vect(x)}\mid x \in \F_2^3 \setminus \Vect(e_1,e_2)\bigr\}.$$
Then Claim \ref{twokinds} clearly entails that $f$ must stabilize $\calX$.
We then lose no generality (left-composing $f$ with $M \mapsto M^T$ if necessary) assuming that there are four lines
$D_1$, $D'_1$, $D_2$, $D'_2$, with $D_1 \neq D_2$, and none of them included in $\F_2^2 \times \{0\}$, such that
$f(\calV_2(\F_2) \cap \calM_{D_1})=\calV_2(\F_2) \cap \calM_{D'_1}$ and
 $f(\calV_2(\F_2) \cap \calM_{D_2})=\calV_2(\F_2) \cap \calM_{D'_2}$.
Claim \ref{inverseimageclaim} then easily follows from Claim \ref{stabX},
and then the rest of Section \ref{strongstart} shows that $f$ extends to a Frobenius automorphism of $\Mat_3(\F_2)$.

\subsection{The case of $\frak{sl}_3(\F_2)$}

Here, we let $f : \frak{sl}_3(\F_2) \rightarrow \frak{sl}_3(\F_2)$ be a bijective linear transformation which preserves non-singularity.
Note again that $f$ is a determinant preserver. Our aim is to prove Claim \ref{inverseimageclaim} in this situation.
This has the following three steps:

\begin{lemme}
No linear subspace of $\frak{sl}_3(\F_2)$ is equivalent to $\calR(1,1)$.
\end{lemme}

\begin{proof}
Indeed, if there were such a linear subspace, then there would be a $2$-dimensional linear subspace $P$
of $\F_2^3$ such that $\frak{sl}_3(\F_2)$ contains every matrix which vanishes on $P$, one of which
has a non-zero trace.
\end{proof}

Using the line of reasoning from Section \ref{strongstart}, it thus suffice to prove the following:

\begin{claim}\label{lastcompatclaim}
For every line $D \subset \F_2^3$, neither $f^{-1}(\calM_D)$ nor $f^{-1}(\calM^D)$ is equivalent to $\calJ_3(\F_2)$.
\end{claim}

To prove this, we establish two lemmas:

\begin{lemme}\label{lastcountlemma}
There are five rank $1$ matrices in $\calJ_3(\F_2)$. \\
For any $D \in \F_2^3$, there are more than five rank $1$ matrices in $\calM_D \cap \frak{sl}_3(\F_2)$,
and the same holds for $\calM^D \cap \frak{sl}_3(\F_2)$.
\end{lemme}

\begin{lemme}\label{sl23rankpreserver}
The map $f$ is a rank preserver.
\end{lemme}

Clearly, combining those lemmas yields Claim \ref{lastcompatclaim}, hence the rest of the proof from Section \ref{strongstart}
applies with no restriction.

\begin{proof}[Proof of Lemma \ref{lastcountlemma}]
The first claim is straightforward (notice that a matrix of $\calJ_3(\F_2)$ has rank $1$ only if its diagonal is zero). \\
For the second one, we lose no generality assuming that $D$ is spanned by the first vector of the canonical basis and
by only considering the case of $\calM_D \cap \frak{sl}_3(\F_2)$.
Then $\calM_D \cap \frak{sl}_3(\F_2)$ is the set of all matrices of the form
$\begin{bmatrix}
0 & L \\
0 & M
\end{bmatrix}$ with $L \in \Mat_{2,1}(\K)$ and $M \in \frak{sl}_2(\F_2)$. Taking $M=0$ and an arbitrary $L \neq 0$
yields three rank $1$ matrices, then taking $L=0$ and an arbitrary nilpotent matrix $M$ yields three others.
\end{proof}

\begin{proof}[Proof of Lemma \ref{sl23rankpreserver}]
We start by using the fact that $f$ is a determinant preserver on $\frak{sl}_3(\F_2)$.
The Newton formulae show that $\tr A^3=3\det A=\det A$ for every $A \in \frak{sl}_3(\F_2)$.
It follows that $f$ preserves the form $b(A,B)=\det(A+B)-\det(A)-\det(B)=\tr(A^2B)+\tr(B^2A)$.
Fixing $A$ and computing $b(A,B+C)-b(A,B)-b(A,C)$ for an arbitrary pair $(B,C)$, we deduce:
\begin{equation}\label{ternary}
\forall (A,B,C)\in \frak{sl}_3(\F_2)^3, \;
\tr\bigl([f(A),f(B)]f(C)\bigr)=\tr\bigl([A,B]C\bigr),
\end{equation}
where $[-,-]$ denotes the standard Lie bracket on $\Mat_3(\F_2)$.
For $A \in \Mat_3(\F_2)$, denote by $\calC(A):=\{M \in \Mat_3(\F_2) : \; [A,M]=0\}$ its centralizer
and set $\calC'(A):=\calC(A)\cap \frak{sl}_3(\F_2)$.
Notice then that identity \eqref{ternary} yields:
$$\forall A \in \frak{sl}_3(\F_2), \; f(\calC'(A))=\calC'(f(A)).$$
Indeed $\tr(I_3^2)=1$, hence the symmetric bilinear form $(A,B) \mapsto \tr(AB)$ is non-degenerate on the orthogonal $\frak{sl}_3(\F_2)$
of $\Vect(I_3)$.
However, since $I_3 \not\in \frak{sl}_3(\F_2)$, we may write $\calC(A)=\Vect(I_3)\oplus \calC'(A)$
for any $A \in \frak{sl}_3(\F_2)$, which yields:
$$\forall A \in \frak{sl}_3(\F_2), \; \dim \calC(A)=\dim \calC(f(A))$$
i.e.\ $f$ preserves the dimension of centralizers. It now suffices to prove that
$f$ preserves the set of rank $1$ matrices of $\frak{sl}_3(\F_2)$.
In order to do this, we characterize the rank $1$ matrices in $\frak{sl}_3(\F_2)$ in terms of their centralizer, in the next lemma.
\end{proof}

\begin{lemme}
Let $A \in \frak{sl}_3(\F_2)$. Then the following conditions are equivalent:
\begin{enumerate}[(i)]
\item $\rk A=1$;
\item $\dim \calC(A)=5$ and $\calC'(A) \cap \frak{sl}_3(\F_2)$ contains only singular matrices.
\end{enumerate}
\end{lemme}

\begin{proof}
Assume $\rk A=1$. Then $A$ is nilpotent since $\tr A=0$, and we lose no generality assuming that $A$ is the elementary matrix $E_{1,3}$.
A straightforward computation then yields
$$\calC(A)=\Biggl\{\begin{bmatrix}
a & b & c \\
0 & d & e \\
0 & 0 & a
\end{bmatrix}\mid (a,b,c,d,e)\in \F_2^5\Biggr\} \quad \text{and} \quad
\calC'(A)=\Biggl\{\begin{bmatrix}
a & b & c \\
0 & 0 & e \\
0 & 0 & a
\end{bmatrix}\mid (a,b,c,e)\in \F_2^4\Biggr\}$$
hence $A$ satisfies condition (ii). \\
Conversely, assume condition (ii) holds and $\rk A \neq 1$.
The condition $\dim \calC(A)=5$ shows, using the Frobenius formula for the dimension of the centralizer (see Theorem 19 p.111 of \cite{Jacobson}),
that $A$ is a linear combination of $I_3$ and a rank $1$ matrix $B$.
However $\rk A \neq 1$ and $A \neq I_3$ hence $A=I_3+B$. We deduce that $\tr B=1$ hence we lose no generality assuming that
$A=\begin{bmatrix}
1 & 0 & 0 \\
0 & 1 & 0 \\
0 & 0 & 0
\end{bmatrix}$. Then the non-singular matrix
$\begin{bmatrix}
0 & 1 & 0 \\
1 & 1 & 0 \\
0 & 0 & 1
\end{bmatrix}$ commutes with $A$ and has trace $0$, which contradicts condition (ii).
\end{proof}

This finishes the proof of Lemma \ref{sl23rankpreserver} and shows that $f$
extends to a Frobenius automorphism of $\Mat_3(\F_2)$.

\subsection{Conclusion}

We may now sum up the previous results:

\begin{theo}\label{finaltheoF2}
Let $V$ be a linear hyperplane of $\Mat_3(\F_2)$ which is not equivalent to $\calV_1(\F_2)$,
and $f : V \hookrightarrow \Mat_3(\F_2)$ be a linear embedding such that
$\forall M \in V, \; f(M) \in \GL_3(\F_2) \Leftrightarrow M \in \GL_3(\F_2)$.
Then $f$ extends to a Frobenius automorphism.
\end{theo}

\section{The case of linear hyperplanes of $\Mat_2(\K)$}\label{hyper}

In this final section, we show that the result from Theorem \ref{strong}
still holds in the case $n=2$ and $V$ is a linear hyperplane of $\Mat_2(\K)$, and we also investigate
the question of weak preservers.
Using the same line of reasoning as in section \ref{redtointernalF2}, we see that, up to equivalence,
the only linear hyperplanes of $\Mat_2(\K)$ are
$T_2^+(\K)$ (the set of upper triangular matrices of $\Mat_2(\K)$)
and $\frak{sl}_2(\K)$. Let $f : V \hookrightarrow \Mat_2(\K)$ be a linear embedding such that
$f(V \cap \GL_2(\K))\subset \GL_2(\K)$. We lose no generality assuming that
both $V$ and $f(V)$ belong to $\bigl\{T_2^+(\K),\frak{sl}_2(\K)\bigr\}$.

\noindent Let us assume first that $V=f(V)$.

\begin{itemize}
\item If $V=\frak{sl}_2(\K)$, then $f$ is a linear automorphism of $\frak{sl}_2(\K)$ which (weakly) preserves nilpotency,
hence the Botta-Pierce-Watkins theorem \cite{BPW} (or classical results on projective conics) shows that
$f$ extends to $u_{\lambda\,P,P^{-1}}$ for some pair $(\lambda,P)\in \K^* \times \GL_2(\K)$.
\item If $V=T_2^+(\K)$, then a theorem of Chooi and Lim \cite{ChoiiLim} shows that $f$ extends to a
Frobenius automorphism of $\Mat_2(\K)$.
\end{itemize}

\noindent Assume now only that $f(V \cap \GL_2(\K))\subset \GL_2(\K)$.

\begin{itemize}
\item If $f^{-1}(\GL_2(\K))=V \cap \GL_2(\K)$, then
$V$ is equivalent to $f(V)$ since $T_2^+(\K)$ contains a $2$-dimensional singular subspace whereas $\frak{sl}_2(\K)$ does not.
For the same reason, $f(V)$ is equivalent to $V$ if $V=\frak{sl}_2(\K)$.
\end{itemize}

\noindent This proves the following results:

\begin{prop}
Let $V$ be a linear hyperplane of $\Mat_2(\K)$, and $f : V \hookrightarrow \Mat_2(\K)$
be a linear embedding such that $f^{-1}(\GL_2(\K))=V \cap \GL_2(\K)$.
Then $f$ extends to a Frobenius automorphism of $\Mat_2(\K)$.
\end{prop}

\begin{prop}\label{lastweakfirst}
Let $V$ be a linear hyperplane of $\Mat_2(\K)$ which is equivalent to $\frak{sl}_2(\K)$, and
$f : V \hookrightarrow \Mat_2(\K)$ be a linear embedding such that $f(V \cap \GL_2(\K))\subset \GL_2(\K)$.
Then $f$ extends to a Frobenius automorphism of $\Mat_2(\K)$.
\end{prop}

Let us finally examine whether there exists a
linear bijective map $f : T_2^+(\K) \rightarrow \frak{sl}_2(\K)$ which
maps non-singular matrices to non-singular matrices. Assume such a map exists. Then
$f^{-1}$ maps any singular matrix of $\frak{sl}_2(\K)$ to a singular matrix.
Denote by $u$ the projective isomorphism from $\Pgros\bigl(\frak{sl}_2(\K)\bigr)$ to
$\Pgros(T_2^+(\K))$ associated to $f^{-1}$. The set of rank $1$ matrices of $\frak{sl}_2(\K)$ yields a non-degenerate
projective conic $\calC$ of $\Pgros\bigl(\frak{sl}_2(\K)\bigr)$ and $u$ maps $\calC$ into
the union of two distinct projective lines $\calD_1$ and $\calD_2$ of $\Pgros(T_2^+(\K))$.
This is impossible if $\# \K \geq 4$ since a non-degenerate projective conic has at most $2$ common points with every line
and $\# \calC=\# \K+1$.
However, if $\# \K \leq 3$, then $\calC$ has at most $4$ points, hence we may find two distinct projective lines $\calD'_1$ and $\calD'_2$
such that $\calC \subset \calD'_1 \cup \calD'_2$: we may then choose a projective transformation $v : \Pgros(\frak{sl}_2(\K)) \rightarrow
\Pgros(T_2^+(\K))$ which maps respectively $\calD'_1$ to $\calD_1$ and $\calD'_2$ to $\calD_2$.
Given a linear map $g : \frak{sl}_2(\K) \rightarrow T_2^+(\K)$ associated to $u$,
we then find that $g^{-1}$ is a weak linear preserver of non-singularity but does not extend to a Frobenius automorphism of $\Mat_2(\K)$.
We may now generalize Proposition \ref{lastweakfirst} as follows:

\begin{prop}\label{lastweakrenforce}
Let $V$ be a linear hyperplane of $\Mat_2(\K)$ and $f : V \hookrightarrow \Mat_2(\K)$ be a linear embedding such that $f(V \cap \GL_2(\K))\subset \GL_2(\K)$.
Then $f$ extends to a Frobenius automorphism of $\Mat_2(\K)$
unless $V$ is equivalent to $T_2^+(\K)$ and $\# \K \leq 3$.
\end{prop}

We conclude by summing up the previous results in the case of a linear hyperplane of $\Mat_n(\K)$
(the case $n=1$ being trivial).

\begin{theo}\label{hyperplansynthese}
Let $V$ be a linear hyperplane of $\Mat_n(\K)$, and $f : V \rightarrow V$
be a linear automorphism such that $f(V \cap \GL_n(\K)) \subset \GL_n(\K)$. \\
Then $f$ extends to a Frobenius automorphism unless $n=3$, $\K \simeq \F_2$
and $V$ is equivalent to $\calV_1(\F_2)$.
\end{theo}

\begin{theo}\label{sllasttheo}
Let $f : \frak{sl}_n(\K) \rightarrow \frak{sl}_n(\K)$
be a linear automorphism such that $f(\frak{sl}_n(\K) \cap \GL_n(\K)) \subset \GL_n(\K)$.
Then there exists $P \in \GL_n(\K)$ and a non-zero scalar $\lambda$ such that
$$\forall M \in \frak{sl}_n(\K), \; f(M)=\lambda\,PMP^{-1} \quad \text{or}
\quad \forall M \in \frak{sl}_n(\K), \; f(M)=\lambda\,PM^T\,P^{-1}.$$
\end{theo}

Note that we find exactly the linear preservers of nilpotency (the common ground being the case $n=2$, as we have just seen)!

\begin{proof}[Proof of Theorem \ref{sllasttheo}]
Using Theorem \ref{hyperplansynthese}, it suffices to show that a Frobenius automorphism which stabilizes
$\frak{sl}_n(\K)$ must be of the aforementioned form. Since $\frak{sl}_n(\K)$ is stable under transposition,
it suffices to fix an arbitrary $(P,Q)\in \GL_n(\K)^2$ such that $u_{P,Q}$ stabilizes $\frak{sl}_n(\K)$
and prove that $Q$ is a scalar multiple of $P^{-1}$. However, for every $M \in \frak{sl}_n(\K)$, one has
$\tr(QPM)=\tr(PMQ)=0$ hence $QP$ is a scalar multiple of $I_n$, which proves our claim.
\end{proof}

\end{document}